\documentclass[a4paper, 10pt, oneside, notitlepage]{amsart}

\usepackage[utf8]{inputenc}
\usepackage{color}
\usepackage{amsmath} 
\usepackage{amssymb} 
\usepackage{amsthm}
\usepackage{geometry}
\usepackage{graphicx}
\usepackage{mathtools}
\usepackage{esint}
\usepackage{comment}
\usepackage[ocgcolorlinks=true,linkcolor=blue]{hyperref}

\theoremstyle{plain}
\newtheorem{thm}{Theorem}
\newtheorem{prop}{Proposition}[section]
\newtheorem{lem}[prop]{Lemma}

\newtheorem{rmk}[prop]{Remark}

\newcommand {\R} {\mathbb{R}} 
 \newcommand {\N} {\mathbb{N}}
 
\newcommand {\p} {\partial}

\newcommand {\D} {\Delta}

\newcommand {\supp} {\text{supp}}

\DeclareMathOperator {\dist} {dist}

\DeclareMathOperator{\F} {\mathcal{F}}

\title[Uniqueness and reconstruction for the fractional Calder\'on problem]{Uniqueness and reconstruction for the fractional Calder\'on problem with a single measurement}

\author[T. Ghosh]{Tuhin Ghosh}
\address{Jockey Club Institute for Advanced Study, HKUST, Hong Kong}
\email{iasghosh@ust.hk}

\author[A. R\"uland]{Angkana R\"uland}
\address{Max-Planck-Institute for Mathematics in the Sciences, Inselstr. 22, 04103 Leipzig}
\email{rueland@mis.mpg.de}

\author[M. Salo]{Mikko Salo}
\address{Department of Mathematics and Statistics, University of Jyv\"askyl\"a}
\email{mikko.j.salo@jyu.fi}

\author[G. Uhlmann]{Gunther Uhlmann}
\address{Department of Mathematics, University of Washington / Jockey Club Institute for Advanced Study, HKUST, Hong Kong}
\email{gunther@math.washington.edu}

\begin{document}

\let\thefootnote\relax\footnotetext{Accepted Manuscript for Journal of Functional Analysis, volume and pages to be assigned by Elsevier, DOI:10.1016/j.jfa.2020.108505. This manuscript version is made available under the CC-BY-NC-ND 4.0 license http://creativecommons.org/licenses/by-nc-nd/4.0} 
\maketitle

\begin{abstract}
We show global uniqueness in the fractional Calder\'on problem with a single measurement and with data on arbitrary, possibly disjoint subsets of the exterior. The previous work \cite{GhoshSaloUhlmann} considered the case of infinitely many measurements. The method is again based on the strong uniqueness properties for the fractional equation, this time combined with a unique continuation principle from sets of measure zero. We also give a constructive procedure for determining an unknown potential from a single exterior measurement, based on constructive versions of the unique continuation result that involve different regularization schemes.
\end{abstract}

\section{Introduction}
\label{sec:introduction}

In this article we show global uniqueness in the fractional Calder\'on problem with a single measurement, and provide a reconstruction algorithm. The fractional Calder\'on problem asks to determine an a priori unknown potential $q$ (in a suitable function space, e.g.\ $q\in L^{\infty}(\Omega)$) from exterior measurements encoded by the Dirichlet-to-Neumann map, formally given by 
\begin{align*}
\Lambda_q: H^s(\Omega_e) \rightarrow (H^s(\Omega_e))^*, \quad
f \mapsto (-\D)^s u|_{\Omega_e},
\end{align*}
where the functions $u, f$ are related through the equation
\begin{align}
\label{eq:frac_Schr}
\begin{split}
((-\D)^s + q) u & = 0 \mbox{ in } \Omega,\\
u & = f \mbox{ in } \Omega_e.
\end{split}
\end{align}
Here $\Omega \subset \R^n$ is a bounded open set, $\Omega_e = \R^n \setminus \overline{\Omega}$ is the exterior domain, and the fractional Laplacian is defined through its Fourier symbol, i.e., $(-\D)^s w := \F^{-1} |\xi|^{2s} \F w$ for $w \in H^s(\R^n)$, where $s$ is a real number with $0 < s < 1$. We will assume the following condition:
\begin{equation} \label{dirichlet_uniqueness}
\left\{ \begin{array}{c} \text{if $u \in H^s(\R^n)$ solves $((-\Delta)^s + q)u = 0$ in $\Omega$ and $u|_{\Omega_e} = 0$,} \\
\text{then $u \equiv 0$.} \end{array} \right.
\end{equation}
This means that zero is not a Dirichlet eigenvalue of $(-\D)^s + q$, and one indeed has a unique solution $u \in H^s(\R^n)$ for any exterior value $f$.

This problem, which was first introduced in \cite{GhoshSaloUhlmann}, should be viewed as a fractional analogue of the classical Calder\'on problem, which is a well-studied inverse problem for which we refer to the survey article \cite{Uhlmann_survey} and the references therein. Due to the results of \cite{GhoshSaloUhlmann}, it is known that the Dirichlet-to-Neumann map uniquely determines the potential $q$, i.e.\ if $q_1,q_2\in L^{\infty}(\Omega)$ are such that zero is not a Dirichlet eigenvalue of $(-\D)^s + q_i$, $i\in\{1,2\}$, then 
\begin{align*}
\Lambda_{q_1} = \Lambda_{q_2} \implies q_1 = q_2.
\end{align*}
Moreover, uniqueness holds if the measurements are made on arbitrary, possibly disjoint subsets of the exterior. In \cite{RS17a} this has further been extended to (almost) optimal function spaces, including potentials in $L^{\frac{n}{2s}}(\Omega)$. Logarithmic stability for this inverse problem was also proved in \cite{RS17a}, and this stability is optimal \cite{RulandSalo_instability}. Uniqueness for recovering a potential in the anisotropic fractional equation $((-\mathrm{div}(A \nabla u))^s + q)u = 0$ was shown in \cite{GhoshLinXiao}, and related inverse problems for the semilinear equation $(-\Delta)^s u + q(x,u) = 0$ were studied in \cite{LaiLin}. A reconstruction method for positive potentials based on monotonicity methods was given in \cite{HarrachLin}. See also the survey article \cite{Salo_fractional_survey}.

All previously mentioned works deal with the case of infinitely many measurements, where one knows $\Lambda_q(f)|_{W_2}$ for all $f \in C^{\infty}_c(W_1)$ for some open subsets $W_j$ of $\Omega_e$. Here, we show that measuring $\Lambda_q(f)|_{W_2}$ for a single (nontrivial) $f \in C^{\infty}_c(W_1)$ is enough to determine the potential. Moreover, we give a constructive procedure for determining $q$ from a single measurement (we refer to Section \ref{sec:fs} for the precise functional set-up and the definition of the relevant function spaces).

\begin{thm} \label{thm_main}
Let $\Omega \subset \R^n$, $n \geq 1$, be a bounded open set, let $0 < s < 1$, and let $W_1, W_2 \subset \Omega_e$ be open sets with $\overline{\Omega} \cap \overline{W}_1 = \emptyset$. 
Assume that either
\begin{itemize}
\item $s\in [\frac{1}{4},1)$ and $q\in L^{\infty}(\Omega)$,
\item or $q \in C^0(\overline{\Omega})$ {\rm{(}}in which case $s\in (0,1)$ can be chosen arbitrarily{\rm{)}},
\end{itemize}
and that \eqref{dirichlet_uniqueness} holds. 
Given any fixed function $f \in \widetilde{H}^s(W_1) \setminus \{ 0 \}$, the potential $q$ is uniquely determined and can be reconstructed from the knowledge $\Lambda_q(f)|_{W_2}$.
\end{thm}

We emphasize that Theorem \ref{thm_main} contains both a uniqueness result, i.e.\ the statement that $\Lambda_{q_1}(f) = \Lambda_{q_2}(f)$ for a single $f\in \widetilde{H}^s(W_1)\setminus \{0\}$ implies that $q_1 = q_2$, \textit{and} an algorithmic reconstruction result, i.e., an explicit method of recovering a potential $q$ constructively from a single function $f\in \widetilde{H}^s(W_1)\setminus \{0\}$ and the single measurement $\Lambda_q(f)$.
We note that this solves a formally well-determined inverse problem in any dimension $n \geq 1$, since we recover a function of $n$ variables (the unknown potential $q$) from a measurement that also depends on $n$ variables (the function $\Lambda_q(f)|_{W_2}$ for a fixed $f$). In contrast, the Schwartz kernel of the full DN map $\Lambda_q$ depends on $2n$ variables. Thus the inverse problem with infinitely many measurements is formally overdetermined in any dimension.

The proof of Theorem \ref{thm_main} is based on the strong uniqueness properties of the fractional equation. These were also crucial in \cite{GhoshSaloUhlmann} and subsequent works, where the uniqueness property was used to prove a strong approximation property of the fractional equation, and the approximation property was then used in solving the inverse problems. Here, in the case of a single measurement, we give a proof that only requires different versions of the uniqueness property. We remark that in the slightly different context of the recovery of an unknown obstacle, it was shown in \cite{CLL17} that in the fractional setting a single measurement suffices to recover the obstacle.

The next result is a constructive version of (a special case of) the uniqueness result stated in \cite[Theorem 1.2]{GhoshSaloUhlmann}.

\begin{thm} \label{thm_ucp_constructive}
Let $\Omega \subset \R^n$, $n \geq 1$, be a bounded open set, let $0 < s < 1$, and let $W$ be an open set with $\overline{\Omega} \cap \overline{W} = \emptyset$. Any function $v \in H^s(\R^n)$ with $\mathrm{supp}(v) \subset \overline{\Omega}$ is uniquely determined by the knowledge of $(-\Delta)^s v|_W =: h$. The function $v$ can be reconstructed from $h$ as 
\[
v = \lim_{\alpha \to 0} v_{\alpha} \qquad \mathrm{(}\text{limit in $H^s(\R^n)$}\mathrm{)},
\]
where $v_{\alpha}$ for any $\alpha > 0$ is the unique solution of the following minimization problem:
\[
v_{\alpha} = \mathrm{arg\,min}_{w \in \widetilde{H}^s(\Omega)} \left[ \| (-\Delta)^s w|_W - h \|_{H^{-s}(W)}^{2} + \alpha \|w\|_{H^s(\R^n)}^2 \right].
\]
\end{thm}
 
The previous theorem is essentially an application of the standard Tikhonov regularization scheme to the unique continuation problem of determining $v$ from the knowledge of $(-\Delta)^s v|_W$. Analogues of the corresponding constructive unique continuation results for the case $s=1$ can for instance be found in \cite{KT04, BD10}.
In Section \ref{sec:Aq} we present two additional schemes, based on spectral regularization and minimal $L^2$ norm regularization, to achieve the same result. We note that the results in \cite{RS17a, RulandSalo_instability} strongly suggest that this unique continuation problem is highly ill-posed and has only logarithmic stability. In Section \ref{sec:stab}, we show that this is indeed the case. 
 
Theorem \ref{thm_ucp_constructive}, combined with an application of the uniqueness result in \cite[Theorem 1.2]{GhoshSaloUhlmann} in $\Omega$, would be sufficient to prove Theorem \ref{thm_main} for potentials in $C^0(\overline{\Omega})$. To deal with potentials in $L^{\infty}(\Omega)$, we also need the following unique continuation result for the fractional equation from sets of positive measure.
 
\begin{thm} \label{thm_ucp_positive_measure}
Let $\Omega \subset \R^n$, $n \geq 1$, be a bounded open set, let $ s \in[\frac{1}{4}, 1)$, and let $q \in L^{\infty}(\Omega)$. If $u \in H^s(\R^n)$ satisfies $((-\Delta)^s + q)u = 0$ in $\Omega$ and $u$ vanishes in a set of positive measure in $\Omega$, then $u \equiv 0$ in $\R^n$.
\end{thm}

This type of result has been proved for $C^1$ potentials in \cite{FF14}. Our proof is based on Carleman estimates and a boundary unique continuation principle for solutions of the degenerate elliptic equation $\nabla \cdot (x_{n+1}^{1-2s} \nabla u) = 0$ in $\R^{n+1}_+$ that satisfy a vanishing Robin boundary condition. The restriction $s \geq 1/4$ is required to deal with a $L^{\infty}$ Robin coefficient (and could be removed if $q$ is $C^1$ in a suitable radial direction). The same restriction also appears in the strong unique continuation principle for fractional equations with $L^{\infty}$ potentials \cite{Rue15}.

Let us conclude by describing the reconstruction procedure in Theorem \ref{thm_main}, which determines the unknown potential $q$ from a single measurement $\Lambda_q(f)|_{W_2} =: g$ corresponding to a fixed exterior Dirichlet data $f \in \widetilde{H}^s(W_1) \setminus \{0\}$. The idea is to determine the solution $u \in H^s(\R^n)$ having exterior data $f$ from the knowledge of $f$ and $g$. In the following procedure, we do this by first writing $u = f + v$ where $v \in \widetilde{H}^s(\Omega)$, and then by determining $v$:

\begin{enumerate}
\item 
Define $h := g - (-\Delta)^s f|_{W_2} \in H^{-s}(W_2)$.
\item 
Determine $v \in H^s(\R^n)$ as $v = \lim_{\alpha \to 0} v_{\alpha}$, where $v_{\alpha}$ for $\alpha > 0$ is obtained by solving the minimization problem in Theorem \ref{thm_ucp_constructive} with $W$ replaced by $W_2$.
\item 
Define $u := f + v \in H^s(\R^n)$.
\item 
Determine $q$ a.e.\ in $\Omega$ as 
\[
q := -\frac{(-\Delta)^s u}{u} \Big|_{\Omega}.
\]
Here we use that $u$ can only vanish in a set of measure zero in $\Omega$, by Theorem \ref{thm_ucp_positive_measure} and the fact that $f \not\equiv 0$.
\end{enumerate}
We note that this reconstruction procedure is quite different from those for the standard Calder\'on problem (the case $s=1$), which are often based on complex geometrical optics solutions and boundary integral equations \cite{Nachman1988, Nachman1996}.

This paper is organized as follows. Section \ref{sec:introduction} is the introduction, and Section \ref{sec:aux} discusses function spaces, wellposedness and functional analysis results required for the proofs. Section \ref{sec:Aq} contains several constructive unique continuation results and in particular proves Theorem \ref{thm_ucp_constructive}. Section \ref{sec:recov} considers the inverse problem and proves Theorem \ref{thm_main}. Section \ref{sec:ucp} contains the unique continuation result from sets of positive measure, Theorem \ref{thm_ucp_positive_measure}, which is required to deal with $L^{\infty}$ coefficients in Theorem \ref{thm_main}. Section \ref{sec:stab} shows that logarithmic stability is optimal in Theorem \ref{thm_ucp_constructive}, and Appendix \ref{sec:append} proves a Carleman estimate required for Theorem \ref{thm_ucp_positive_measure}.

\subsection*{Acknowledgements}
M.S.\ was supported by the Academy of Finland (Centre of Excellence in Inverse Modelling and Imaging, grant numbers 312121 and 309963) and by the European Research Council under FP7/2007-2013 (ERC StG 307023) and Horizon 2020 (ERC CoG 770924). G.U.\ was partly supported by NSF and a Si-Yuan Professorship at IAS, HKUST.

\section{Auxiliary Results}
\label{sec:aux}

In this section, we recall a number of auxiliary results, which will be relevant in our reconstruction algorithm.

\subsection{Function spaces}
\label{sec:fs}

In the sequel, we will use several $L^2$ based Sobolev spaces. Here we follow the notation from \cite{RS17a}, \cite{GhoshSaloUhlmann} and \cite{McLean}. The whole space Sobolev spaces are denoted by
\begin{align*}
H^{s}(\R^n):=\{u \in \mathcal{S}'(\R^n): \|\langle D\rangle^{s} u \|_{L^2(\R^n)}<\infty\},
\end{align*}
where $\langle D\rangle^s u := \mathcal{F}^{-1} \{ (1+|\xi|^2)^{s/2} \F u \}$ and where $\F$ denotes the Fourier transform.

For spaces on open domains $U \subset \R^n$ we use the following notation:
\begin{align*}
H^{s}(U)&:= \{u|_{U}: \ u \in H^{s}(\R^n)\},\\
\widetilde{H}^s(U)&:=\mbox{closure of $C_c^{\infty}(U)$ in $H^{s}(\R^n)$},\\
H^{s}_{0}(U)&:=\mbox{closure of $C_c^{\infty}(U)$ in $H^{s}(U)$},\\
H^{s}_{\overline{U}}&:=\{ u \in H^{s}(\R^n): \supp(u) \subset \overline{U}\}.
\end{align*}
We remark that it always holds that (see for instance Theorem 3.3 in \cite{CWHM})
\begin{align*}
(H^{s}(U))^{\ast}= \widetilde{H}^{-s}(U), \ (\widetilde{H}^{s}(U))^{\ast}= H^{-s}(U), \ s \in \R.
\end{align*}
If in addition $U$ is a bounded Lipschitz domain, we also have that 
\begin{align*}
H^{s}_{\overline{U}} &= \widetilde{H}^s(U), \ s \in \R,\\
H^{s}_0(U) &= H^{s}_{\overline{U}}, \ s>-\frac{1}{2}, \ s \notin\{\frac{1}{2}, \frac{3}{2}, \cdots \},\\
H^{s}_0(U)&= H^s(U), \ s \leq \frac{1}{2}.
\end{align*}

\subsection{Well-posedness}
We recall the main well-posedness results for solutions to 
\begin{align}
\label{eq:eq_q}
\begin{split}
((-\D)^s + q)u & = 0 \mbox{ in } \Omega,\\
u & = f \mbox{ in } \Omega_e.
\end{split}
\end{align}
Here and in the remainder of the article, we always implicitly work under the assumption \eqref{dirichlet_uniqueness}. As the well-posedness of \eqref{eq:eq_q} was discussed in detail in \cite{GhoshSaloUhlmann}, we omit the proofs in the sequel and only state the main results.

We first recall the well-posedness in the energy space:

\begin{lem}[Lemma 2.3 in \cite{GhoshSaloUhlmann}]
\label{lem:well_posed}
Let $n\geq 1$ and $s\in (0,1)$.
Let $\Omega \subset \R^n$ be a bounded open set. Assume further that $q\in L^{\infty}(\Omega)$ is such that \eqref{dirichlet_uniqueness} is satisfied.
Let
\begin{align*}
B_q(u,w):=((-\D)^{s/2}u, (-\D)^{s/2} w)_{L^2(\R^n)} + (q u|_{\Omega}, w|_{\Omega})_{L^2(\Omega)}, \quad u,w \in H^{s}(\R^n).
\end{align*}
Then, for any $f\in H^s(\R^n)$ the problem \eqref{eq:eq_q} is well-posed in the sense that there exists a unique solution $u \in H^{s}(\R^n)$ with
\begin{align*}
B_q(u,w) = 0 \mbox{ for all } w \in \widetilde{H}^s(\Omega),
\end{align*}
and $u-f \in \widetilde{H}^s(\Omega)$.
Moreover, there exists a constant $C>0$ depending on $n,s,\Omega,q$ such that
\begin{align*}
\|u\|_{H^{s}(\R^n)} \leq C \|f\|_{H^{s}(\R^n)}.
\end{align*}
\end{lem}

In particular, the well-posedness result of Lemma \ref{lem:well_posed} allows us to define the Poisson operator 
\begin{align*}
P_q : \widetilde{H}^{s}(\Omega_e) \rightarrow H^{s}(\R^n), \ f \mapsto u,
\end{align*}
where $u$ is the unique solution to \eqref{eq:eq_q}.

With the bilinear form $B_q(u,v)$ at hand, it is possible to precisely define the Dirichlet-to-Neumann map associated with the fractional Calder\'on problem (cf.\ Lemma 2.4 in \cite{GhoshSaloUhlmann}). To this end, let $[f],[g] \in H^{s}(\R^n)/ \widetilde{H}^{s}(\Omega)=:X$. If $\Omega$ is a Lipschitz domain, the quotient space $X$ can be identified with $H^{s}(\Omega_e)$. Due to this, we will simply write $f$ instead of $[f]$.
The (weak) Dirichlet-to-Neumann map associated with \eqref{eq:eq_q} could be defined as
\begin{align}
\label{eq:DN}
\begin{split}
&\widetilde{\Lambda}_q: X \rightarrow X^{\ast}, \  \langle \widetilde{\Lambda}_q [f],[g] \rangle_{X\rightarrow X^{\ast}} =
B_q(u, g),
\end{split}
\end{align}
where $u$ is a solution to \eqref{eq:eq_q} with data $f$, $B_q(\cdot, \cdot)$ denotes the bilinear form from Lemma \ref{lem:well_posed} and where $\langle \cdot, \cdot \rangle_{X \rightarrow X^{\ast}} $ denotes the duality pairing between the respective spaces.

We will also consider the pointwise Dirichlet-to-Neumann map 
\[
\Lambda_q: \widetilde{H}^s(\Omega_e) \to H^{-s}(\Omega_e), \ \ f \mapsto (-\Delta)^s u|_{\Omega_e}.
\]
This is well defined for any bounded open set $\Omega \subset \R^n$ and any $q \in L^{\infty}(\Omega)$ if \eqref{dirichlet_uniqueness} holds. It was proved in \cite[Lemma 3.1]{GhoshSaloUhlmann} that, if one assumes more regularity for $\Omega$, $q$, and $f$, one has 
\[
\widetilde{\Lambda}_q f = \Lambda_q f.
\]
In this article we will use the pointwise Dirichlet-to-Neumann map $\Lambda_q$, since it directly leads to a reconstruction procedure from a single measurement.

\subsection{Relating the Poisson operator and the Dirichlet-to-Neumann map}

Our reconstruction procedure for the inverse problem boils down to determining a solution $u = P_q f$ in $\R^n$ from the knowledge of $f$ in $\Omega_e$ and $\Lambda_q f|_W$ for some open $W \subset \Omega_e$. Thus, we wish to determine $P_q f$ from $\Lambda_q f|_W$. Since $\Lambda_q f|_W = (-\Delta)^s u|_W$, the problem reduces to determining $u$ in $\R^n$ from the knowledge of $u$ in $\Omega_e$ and $(-\Delta)^s u$ in $W$. Writing $u = f + v$, it is sufficient to determine a function $v \in \widetilde{H}^s(\Omega)$ from the knowledge of $(-\Delta)^s v|_W$. In other words, we need to determine $v$ from $Lv$, where $L$ is the operator introduced in the following lemma.

\begin{lem}
\label{lem:Aq_Lq1}
Let $n \geq 1$ and $s \in (0,1)$. Let $\Omega \subset \R^n$ be a bounded open set, and let $W \subset \R^n$ be an open set with $\overline{\Omega}\cap \overline{W}= \emptyset$. Consider the operator 
\begin{align}
\label{eq:L}
L: \widetilde{H}^{s}(\Omega) \rightarrow H^{-s}(W),\ \ 
v  \mapsto (-\D)^s v|_{W}.
\end{align}
Then $L$ is a compact, injective operator with dense range. 
In particular, there exist orthonormal bases $\{\varphi_j\}_{j = 1}^{\infty}$ of $H^{-s}(W)$ and $\{\psi_j\}_{j = 1}^{\infty}$ of $\widetilde{H}^{s}(\Omega)$ (with respect to the Hilbert space structure of these spaces, see \cite{CWHM}) and singular values $\sigma_j>0$ such that
\begin{align}
\label{eq:L_spec}
L \psi_j = \sigma_j \varphi_j, \ L^{\ast} \varphi_j = \sigma_j \psi_j.
\end{align}
\end{lem}

\begin{proof}
Let $\chi_1, \chi_2 \in C^{\infty}_c(\R^n)$ satisfy $\chi_2 = 1$ near $\overline{\Omega}$, $\chi_1 = 1$ near $\overline{W}$ and $\chi_1 = 0$ near $\mathrm{supp}(\chi_2)$. Then 
\[
Lv = r_W \chi_1 (-\Delta)^s \chi_2 v, \qquad v \in \widetilde{H}^s(\Omega),
\]
where $r_{W}$ denotes the restriction to $W$. Next we claim that 
\begin{equation} \label{pseudolocality}
\text{$\chi_1 (-\Delta)^s \chi_2$ is bounded $H^s(\R^n) \to L^2(\R^n)$.}
\end{equation}
From \eqref{pseudolocality} and the compact Sobolev embedding $L^2(B) \subset H^{-s}(B)$ where $B$ is a ball containing $\supp(\chi_1)$, we see that $L: \widetilde{H}^{s}(\Omega) \rightarrow H^{-s}(W)$ is compact. To prove \eqref{pseudolocality} we split $(-\Delta)^s = A_1 + A_2$ where $A_1$ corresponds to the Fourier multiplier $\chi(\xi)|\xi|^{2s}$ and $A_2$ corresponds $(1-\chi(\xi))|\xi|^{2s}$, where $\chi(\xi)$ is supported in $B_1(0)$ and is identically equal to one in $B_{1/2}(0)$. Then $\chi_1 A_1 \chi_2$ is bounded $H^s(\R^n) \to L^2(\R^n)$  since $A_1$ is bounded on $L^2(\R^n)$. Moreover, $A_2$ is a standard pseudodifferential operator, and since $\chi_1$ and $\chi_2$ have disjoint supports the integral kernel of $\chi_1 A_2 \chi_2$ is $C^{\infty}$ (this is the pseudolocal property) and hence this operator is bounded $H^s(\R^n) \to H^t(\R^n)$ for any $t > 0$. This shows \eqref{pseudolocality}.

We have proved that $L$ is compact. Also, $L$ is injective by the weak unique continuation property for the fractional Laplacian \cite[Theorem 1.2]{GhoshSaloUhlmann}. By the Hahn-Banach theorem, to prove the density of the range of $L$ in $H^{-s}(W)$, it suffices to show that the only function $f\in \widetilde{H}^{s}(W) = (H^{-s}(W))^*$ which satisfies
\begin{align*}
(L v, f) = 0 \mbox{ for all } v \in \widetilde{H}^s(\Omega),
\end{align*}
is the zero function. To observe this, note that for any $v \in C^{\infty}_c(\Omega)$, the definition of the duality between $H^{-s}(W)$ and $\widetilde{H}^s(W)$ gives 
\begin{align*}
0 = (Lv, f) = ((-\D)^s v, f)_{\R^n} = (v, (-\Delta)^s f)_{\R^n}.
\end{align*}
Since this is true for all $v \in C^{\infty}_c(\Omega)$, it follows that $(-\Delta)^s f|_{\Omega} = 0$. But also $f|_{\Omega} = 0$, and using again \cite[Theorem 1.2]{GhoshSaloUhlmann} yields that $f \equiv 0$. This concludes the proof of the density result. The rest of the statements follow from the spectral theorem for compact operators.
\end{proof}

We remark that the compactness of $L$ indicates that the recovery of $P_q f$ from $\Lambda_q f|_W$ by inverting the relation $L$ from Lemma \ref{lem:Aq_Lq1} is necessarily ill-posed (cf. Section \ref{sec:stab} for more on the stability properties).

\subsection{Equivalence of Runge approximation and weak unique continuation}

Last but not least, we show that the approximation property and the (weak) unique continuation property used in \cite{GhoshSaloUhlmann} are in fact equivalent. A quantitative version of this equivalence was presented in Lemma 3.3 in \cite{RS17a}. For elliptic second order operators, this equivalence was already proved by Lax \cite{Lax}.

\begin{prop}
Let $\Omega \subset \R^n$, $n\geq 1$, be a bounded open set, and assume that $W\subset \Omega_e$ is open. Let $s\in (0,1)$ and let $q \in L^{\infty}(\Omega)$ satisfy \eqref{dirichlet_uniqueness}. Then the following statements are equivalent:
\begin{itemize}
\item[(i)] For every $\epsilon>0$ and every $v \in L^2(\Omega)$ there exists $f\in \widetilde{H}^{s}(W)$ such that
\begin{align*}
\|v-P_q f\|_{L^2(\Omega)} \leq \epsilon.
\end{align*}
\item[(ii)] Let $v\in L^2(\Omega)$ and assume that $w \in \widetilde{H}^{s}(\Omega)$ is a solution to 
\begin{align}
\label{eq:dual_a}
\begin{split}
((-\D)^s + q) w & = v \mbox{ in } \Omega,\\
w & = 0 \mbox{ in } \Omega_e.
\end{split}
\end{align}
Assume that $(-\D)^s w = 0$ in $W$. Then, $v \equiv 0$ and $w\equiv 0$.
\end{itemize}
\end{prop}

For completeness, we briefly recall the short proof of this.

\begin{proof}
The implication (ii) $\Rightarrow$ (i) follows from the Hahn-Banach theorem as explained in \cite{GhoshSaloUhlmann}. Indeed, (i) is equivalent to the density of $\{ P_q f|_{\Omega} \,;\, f\in \widetilde{H}^s(W) \}$ in $L^2(\Omega)$. Assume that a function $v_0\in L^2(\Omega)$ satisfies 
\begin{align*}
(v_0, P_q f)_{\Omega}= 0 \mbox{ for all } f\in \widetilde{H}^{s}(W).
\end{align*}
Defining $w$ to be a solution of \eqref{eq:dual_a} for $v_0$ then yields (after using the equation for $P_q f$)
\begin{align*}
0 &=(v_0, P_q f)_{\Omega} = (((-\D)^s+q)w, P_q f - f)_{\Omega} = (((-\D)^s+q)w,P_q f - f)_{\R^n} \\ 
 &= -((-\D)^s w, f)_{\R^n} \mbox{ for all } f \in \widetilde{H}^s(W).
\end{align*}
In particular, $(-\D)^s w = 0$ in $W$. Assuming the validity of (ii) hence entails that $v_0\equiv 0$ and $w\equiv 0$, which yields the desired density result.

The opposite implication (i) $\Rightarrow$ (ii) is a consequence of an argument which is similar to the one for Lemma 3.3 in \cite{RS17a}. Let $v\in L^2(\Omega)$ be such that for the solution $w\in \widetilde{H}^{s}(\Omega)$ of \eqref{eq:dual_a} we have $(-\D)^s w|_{W}=0$. Assume that the approximation property from (i) holds. We seek to show that then $v\equiv 0$ and hence $w\equiv 0$. Using the approximation property, we have that for any $\psi \in L^2(\Omega)$ and any $\epsilon>0$ there exists $f\in \widetilde{H}^s(W)$ such that $\|\psi- P_q f\|_{L^2(\Omega)}\leq \epsilon$. Thus, using the equations for $w$ and $P_q f$ as in the first part of this proof and the assumption that $(-\D)^s w|_{W}=0$, we infer that
\begin{align*}
(v,\psi)_{\Omega} & = (v, \psi - P_q f)_{\Omega} + (v, P_q f)_{\Omega}\\
& = (v, \psi - P_q f)_{\Omega} - ((-\D)^s w, f)_{W}
 = (v, \psi - P_q f)_{\Omega}.
\end{align*}
Thus, using the approximation property for any $\psi \in L^2(\Omega)$ and any $\epsilon>0$ we obtain
\begin{align*}
|(v,\psi)_{\Omega}| 
\leq \|v\|_{L^2(\Omega)} \|\psi- P_q f\|_{L^2(\Omega)}
\leq \epsilon \|v\|_{L^2(\Omega)}.
\end{align*}
Letting $\epsilon \rightarrow 0$, we in particular obtain $(v,\psi)_{\Omega}=0$ for all $\psi \in L^2(\Omega)$. Hence $v=0$, which by well-posedness of the equation \eqref{eq:dual_a} also implies that $w\equiv 0$. This concludes the proof.
\end{proof}

\section{Constructive Unique Continuation Results}
\label{sec:Aq}

Seeking to follow the recovery strategy outlined in steps (1)--(4) in the introduction, we here deal with constructive unique continuation results which are needed for step (2). As the operator $L$ from \eqref{eq:L} is compact, this is an ill-posed problem and hence requires regularization arguments. In the sequel, we discuss three such possible recovery procedures: First, we rely on the spectral properties of the operator $L$ from Lemma \ref{lem:Aq_Lq1} and apply a suitable spectral regularization scheme. Next, in Section \ref{sec:Tik}, we rely on Tikhonov regularization and hence prove Theorem \ref{thm_ucp_constructive}. Finally, in Section \ref{sec:minimal} we use a variational argument as in \cite{RS17a} and \cite{FernandezCaraZuazua}, which yields the minimal $L^2$ norm regularization. If the data are exactly of the form $(-\D)^s u|_{W}$ for some function $u\in \widetilde{H}^s(\Omega)$, all these schemes recover $u$ exactly (a little care is needed for this in the minimal $L^2$ norm regularization). However, we remark that in view of the stability results from \cite{RS17a}, the stability for these recovery schemes is at best logarithmic, which renders them very unstable (cf.\ Section \ref{sec:stab}). We will not discuss here the choice of the regularization parameter or computational implementations.

\subsection{Spectral regularization}
\label{sec:spec_reg}

We begin by discussing the spectral regularization, which is based on the mapping properties of $L, L^{\ast}$ outlined in Lemma \ref{lem:Aq_Lq1}. Further properties of spectral regularization can for instance be found in Chapter 4 in \cite{CK} (cf.\ also Section 3.3 in \cite{RS17a}).

\begin{lem}
\label{lem:spec_real}
Let $n\geq 1$ and $s\in(0,1)$.
Let $\Omega \subset \R^n$ be a bounded open set and assume that $W \subset \R^n$ is open with $\overline{\Omega}\cap \overline{W}= \emptyset$.
Let $L, L^{\ast}$ as well as $\{\psi_k\}_{k=1}^{\infty} \subset \widetilde{H}^s(\Omega)$, $\{\varphi_k\}_{k=1}^{\infty} \subset H^{-s}(W)$, $\sigma_k>0$ be as in \eqref{eq:L_spec} and let $h \in H^{-s}(W)$. 
Then, the following approximation results hold:
\begin{itemize}
\item[(i)] The function
\begin{align*}
v_{\alpha}:=R_{\alpha}(h):=\sum\limits_{\sigma_k \geq \alpha} \frac{1}{\sigma_k} (h, \varphi_k)_{H^{-s}(W)} \psi_k \in \widetilde{H}^s(\Omega)
\end{align*} 
satisfies $L v_{\alpha} \rightarrow h$ in $H^{-s}(W)$ as $\alpha \to 0$. 
\item[(ii)] If $h= L v$ for some $v \in \widetilde{H}^s(\Omega)$, we further have 
$R_{\alpha}(h)=v_{\alpha} \rightarrow v$ in $\widetilde{H}^s(\Omega)$.
\end{itemize} 
\end{lem}

\begin{proof}
The claim in (i) directly follows from the density and the mapping properties from Lemma \ref{lem:Aq_Lq1}. 
In order to deduce the second property, we use the (Hilbert space) duality between $L, L^{\ast}$: If $h = L v$, then
\begin{align*}
v_{\alpha} 
&= \sum\limits_{\sigma_k \geq \alpha} \frac{1}{\sigma_k} (h, \varphi_k)_{H^{-s}(W)} \psi_k
= \sum\limits_{\sigma_k \geq \alpha} \frac{1}{\sigma_k} (Lv, \varphi_k)_{H^{-s}(W)} \psi_k\\
&= \sum\limits_{\sigma_k \geq \alpha} \frac{1}{\sigma_k} (v, L^{\ast} \varphi_k)_{\widetilde{H}^s(\Omega)} \psi_k 
= \sum\limits_{\sigma_k \geq \alpha}(v, \psi_k)_{\widetilde{H}^s(\Omega)} \psi_k \\
&\rightarrow
 \sum\limits_{k \geq 1}  (v, \psi_k)_{\widetilde{H}^s(\Omega)}  \psi_k = v \mbox{ in } \widetilde{H}^s(\Omega).
\end{align*}
This concludes the argument.
\end{proof}

\begin{rmk}
\label{rmk:compute} 
We remark that the spectral regularization scheme outlined in Lemma \ref{lem:spec_real} requires the knowledge of $\sigma_k, \psi_k, \varphi_k$. These however can be computed from the eigenvalues and (generalized) eigenfunctions of the (known) operators $L^{\ast} L$ and  $L L^{\ast}$.
\end{rmk}

\subsection{Tikhonov regularization}
\label{sec:Tik}

As a second regularization procedure with possibly less computational effort (it is for instance not needed to first compute the singular value decomposition of $L, L^{\ast}$) we describe a Tikhonov regularization scheme for our problem. Tikhonov regularization is discussed e.g.\ in \cite{CK} (cf.\ also Section 3.3 in \cite{RS17a}).

\begin{lem}
\label{lem:ThikonovII}
Let $n\geq 1$ and $s\in(0,1)$.
Let $\Omega \subset \R^n$ be a bounded open set and assume that $W \subset \R^n$ is open with $\overline{\Omega}\cap \overline{W}= \emptyset$. Then the following results hold:
\begin{itemize}
\item[(i)] For each $h \in H^{-s}(W)$ and each $\alpha>0$ the functional 
\begin{align*}
\mathcal{E}_{\alpha}(v):=\|(-\D)^s v|_{W} - h\|_{H^{-s}(W)}^2 + \alpha \|v\|_{\widetilde{H}^s(\Omega)}^2, \ v \in \widetilde{H}^s(\Omega),
\end{align*}
has a unique minimizer $v_{\alpha}=:R_{\alpha}(h) \in \widetilde{H}^s(\Omega)$. Moreover, $(-\D)^s v_{\alpha}|_{W} \rightarrow h$ in $H^{-s}(W)$ as $\alpha \rightarrow 0$.
\item[(ii)]
If $h \in R(L)$ with $L$ being the operator from \eqref{eq:L}, i.e. if there exists $w \in \widetilde{H}^s(\Omega)$ such that $(-\D)^s w|_{W}=h$, then $v_{\alpha}=R_{\alpha}(h) \rightarrow w$ in $\widetilde{H}^s(\Omega)$.
\end{itemize}
\end{lem}

\begin{proof}
Both properties follow from general arguments on Tikhonov regularization combined with the mapping properties of $L$: As the operator $L$ is a compact, linear operator by Lemma \ref{lem:Aq_Lq1}, Theorem 4.14 in \cite{CK} asserts the existence of a unique solution of the minimization problem for $\mathcal{E}_{\alpha}$. Since furthermore the operator $L$ has a dense image in $H^{-s}(W)$, we also obtain the approximation property claimed in (i) (Theorem 4.15 in \cite{CK}). 

Theorem 4.13 in \cite{CK} implies that Tikhonov regularization is a regularization scheme. Hence, if $h = Lw$ for some $w \in \widetilde{H}^s(\Omega)$, this in particular implies the pointwise convergence 
\begin{align*}
v_{\alpha}= R_{\alpha}(h) \rightarrow w \mbox{ in } \widetilde{H}^s(\Omega)
\end{align*}
as $\alpha \rightarrow 0$.
\end{proof}

\subsection{Minimal $L^2$ norm regularization}
\label{sec:minimal}

Finally, as a further possible means of recovering $v$ from $(-\D)^s v|_{W}$, we use a variational approach which is analogous to the one presented in \cite{RulandSalo_nonlocal} and \cite{FernandezCaraZuazua}. In the sequel, with slight abuse of notation, we will write $\int\limits_{W} h f dx$ also for $h \in H^{-s}(W)$ and $f\in \tilde{H}^s(W)$ to denote the corresponding duality pairing (which we view as an extension by continuity to these spaces).

\begin{lem}
\label{lem:energy}
Let $n\geq 1$ and $s\in(0,1)$.
Let $\Omega \subset \R^n$ be a bounded open set and assume that $W \subset \Omega_e$ is open with $\overline{\Omega}\cap \overline{W}= \emptyset$.
For $\alpha>0$ and $h\in H^{-s}(W)$ consider the functional
\begin{align*}
\mathcal{J}_{\alpha}(f):= \frac{1}{2}\|u\|_{L^2(\Omega)}^2 - \int\limits_{W} h f dx + \alpha \|f\|_{\widetilde{H}^{s}(W)}, \ f \in \widetilde{H}^{s}(W),
\end{align*}
where $u$ and $f$ are related through 
\begin{align}
\label{eq:eq}
\begin{split}
(-\D)^s u & = 0 \mbox{ in } \Omega,\\
u & = f \mbox{ in } \Omega_e.
\end{split}
\end{align}
Then, for each $h \in H^{-s}(W)$ and $\alpha>0$ there exists a unique minimizer $f_{\alpha} \in \widetilde{H}^{s}(W)$ of the functional $\mathcal{J}_{\alpha}$. Denoting the associated solution of \eqref{eq:eq} with exterior data $f_{\alpha}$ by $\hat{u}_{\alpha} \in H^{s}(\R^n)$ and defining $\hat{\varphi}_{\alpha}=:R_{\alpha}(h) \in \widetilde{H}^s(\Omega)$ to be the solution to the dual equation
\begin{align}
\label{eq:dual_1}
\begin{split}
(-\D)^s \hat{\varphi}_{\alpha} & = -\hat{u}_{\alpha} \mbox{ in } \Omega,\\
\hat{\varphi}_{\alpha} & = 0 \mbox{ in } \Omega_e,
\end{split}
\end{align}
we then have $\|(-\D)^s \hat{\varphi}_{\alpha}|_{W} - h \|_{H^{-s}(W)}\leq \alpha$ and $\mathcal{J}_{\alpha}(f_{\alpha}) = -\frac{1}{2} \| \hat{u}_{\alpha} \|_{L^2(\Omega)}^2$.
\end{lem}

\begin{proof}
The proof follows along the lines of Lemmas 4.1 and 4.2 in \cite{RulandSalo_nonlocal}, which is based on the variational approach from \cite{FernandezCaraZuazua}. For self-containedness, we repeat the argument: Firstly, we note that the functional $\mathcal{J}_{\alpha}$ is strictly convex and continuous with respect to $\widetilde{H}^{s}(W)$ convergence (of $f\in \widetilde{H}^{s}(W)$). The proof of strict convexity uses weak unique continuation. Hence, in order to prove existence, it suffices to check coercivity, which follows from the unique continuation property of the fractional Laplacian. Indeed, assume that $f_k \in \widetilde{H}^s(W)$ is a sequence with $\|f_k\|_{\widetilde{H}^s(W)} \rightarrow \infty$. Then, we define the rescaled functions $\hat{f}_k:= \frac{f_k}{\|f_k\|_{\widetilde{H}^s(W)}}$, which are of unit norm (and thus weakly precompact). Rescaling the functional $\mathcal{J}_{\alpha}$ yields
\begin{align}
\label{eq:coerc}
\frac{\mathcal{J}_{\alpha}(f_k)}{\|f_k\|_{\widetilde{H}^s(W)}}
= \frac{1}{2}\|\hat{u}_k\|_{L^2(\Omega)}^2 \|f_k\|_{\widetilde{H}^s(W)} - \int\limits_{W} h \hat{f}_k dx + \alpha .
\end{align}
Here $\hat{u}_k:= \frac{u_k}{\|f_k\|_{\widetilde{H}^s(W)}}$, with $u_k$ being a solution to \eqref{eq:eq} with exterior data $f_k$, and the integral over $W$ denotes the duality of $H^{-s}(W)$ and $\widetilde{H}^s(W)$.
We now distinguish two cases:
\begin{itemize}
\item If $\liminf\limits_{k\rightarrow \infty}\|\hat{u}_k\|_{L^2(\Omega)}>0$, then the boundedness of $\|\hat{f}_k\|_{\widetilde{H}^s(W)}$ and of $\|h\|_{H^{-s}(W)}$ and the unboundedness of $\|f_k\|_{\widetilde{H}^s(W)}$ imply
\begin{align*}
\liminf\limits_{k\rightarrow \infty} \frac{\mathcal{J}_{\alpha}(f_k)}{\|f_k\|_{\widetilde{H}^s(W)}} \geq \liminf\limits_{k\rightarrow \infty} \left[ \frac{1}{2}\|\hat{u}_k\|_{L^2(\Omega)}^2 \|f_k\|_{\widetilde{H}^{s}(W)} - \|h\|_{H^{-s}(W)} + \alpha \right] = \infty.
\end{align*}
This in particular yields the desired coercivity.
\item If along some subsequence in $k$ (which without loss of generality, we may assume to be the whole sequence), we have 
$\lim\limits_{k\rightarrow \infty}\|\hat{u}_k\|_{L^2(\Omega)}=0$, we infer that
\begin{align*}
\hat{u}_k \rightharpoonup \psi \mbox{ in } H^s(\R^n),\ 
\hat{u}_k \rightarrow 0 \mbox{ in } L^2(\Omega), \ \hat{f}_k \rightharpoonup f_{\infty} \mbox{ in } \widetilde{H}^{s}(W).
\end{align*}
Here we used the estimate $\|u_k\|_{H^s(\R^n)} \leq C \|f_k\|_{\widetilde{H}^s(W)}$ for solutions to the equation \eqref{eq:eq} in order to infer the first convergence result.
Moreover, the function $\psi$ in $H^s(\R^n)$ solves
\begin{align*}
(-\D)^s \psi & = 0 \mbox{ in } \Omega,\\
\psi& = f_{\infty} \mbox{ in } \Omega_e.
\end{align*}
By the fact that $\psi= 0$ in $\Omega$ (which follows since $\hat{u}_k \rightarrow 0 \mbox{ in } L^2(\Omega)$) and by (weak) unique continuation for the fractional Laplacian, this however entails that $\psi \equiv 0$. In particular, $f_{\infty}=0$. Thus, returning to \eqref{eq:coerc} and using that $\int\limits_{W} h \hat{f}_k dx \to 0$, we deduce that for $k$ sufficiently large it holds that
\begin{align*}
\frac{\mathcal{J}_{\alpha}(f_k)}{\|f_k\|_{\widetilde{H}^s(W)}} \geq \frac{1}{2}\|\hat{u}_k\|_{L^{2}(\Omega)}^2 \|f_k\|_{\widetilde{H}^s(W)} + \frac{\alpha}{2}
\geq \frac{\alpha}{2}.
\end{align*}
Again, this yields the desired coercivity and therefore concludes the existence proof.
\end{itemize}

The smallness condition $\|(-\D)^s \hat{\varphi}_{\alpha}|_{W} - h \|_{H^{-s}(W)}\leq \alpha$ follows from considering variations of the functional around the minimum. Indeed, spelling out minimality condition
\begin{align*}
\mathcal{J}_{\alpha}(f_{\alpha}) \leq \mathcal{J}_{\alpha}(f_{\alpha} + \mu f), \ \mu \in \R,
\end{align*}
as in \cite{RulandSalo_nonlocal} and combining it with the triangle inequality gives
\begin{align}
\label{eq:ELE}
\left| \int\limits_{\Omega} \hat{u}_{\alpha} u dx - \int\limits_{W} h f dx \right| \leq \alpha \|f\|_{\widetilde{H}^{s}(W)}.
\end{align}
By the definition of $\hat{\varphi}_{\alpha}$ and the identity
\begin{align}
\label{eq:ident_dual}
 \int\limits_{\Omega} \hat{u}_{\alpha} u \,dx = \int\limits_{\R^n} -(-\Delta)^s \hat{\varphi}_{\alpha} (u-f) \,dx =  \int\limits_{W} (-\D)^s \hat{\varphi}_{\alpha}|_W f \,dx, 
\end{align}
we further deduce that
\begin{align}
\label{eq:ELG}
\left|\int\limits_{W} \left[ (-\D)^s \hat{\varphi}_{\alpha}|_{W} -h \right] f dx\right| \leq \alpha \|f\|_{\widetilde{H}^s(W)}.
\end{align}
Duality then implies the estimate $\|(-\D)^s \hat{\varphi}_{\alpha}|_{W} - h \|_{H^{-s}(W)}\leq \alpha$. The condition $\mathcal{J}_{\alpha}(f_{\alpha}) = -\frac{1}{2} \| \hat{u}_{\alpha} \|_{L^2(\Omega)}^2$ follows also from the minimality condition as in \cite{RulandSalo_nonlocal}.
\end{proof}

Having established approximate recovery, we seek to show that the variational argument from above is a regularization scheme, i.e.\ that it recovers the function exactly if $h \in R(L)$, where $L$ is the operator from \eqref{eq:L}. To this end, we will need to assume some extra regularity.

\begin{lem}
\label{lem:functional_0}
Assume the conditions in Lemma \ref{lem:energy}, and assume additionally that $\Omega$ has $C^{\infty}$ boundary. Let $h\in H^{-s}(W)$ and assume that $h\in R(L)$ with $L$ as in \eqref{eq:L}, i.e., that there exists $\overline{\varphi}\in \widetilde{H}^{s}(\Omega)$ with $h = (-\D)^s \overline{\varphi}|_{W}$. Suppose moreover that $\overline{u}:=(-\D)^s\overline{\varphi}|_{\Omega} \in L^2(\Omega)$.
Then, with $\hat{\varphi}_{\alpha} = R_{\alpha} h$, for some sequence $\alpha \to 0$ one has 
\[
\hat{\varphi}_{\alpha} \to \overline{\varphi} \mbox{ in } \widetilde{H}^s(\Omega).
\]
\end{lem}

\begin{proof}
Using the fact that $h = (-\D)^s \bar{\varphi}|_{W}$ and the regularity assumption $\overline{u} =(-\D)^s\overline{\varphi}|_{\Omega} \in L^2(\Omega)$, we compute 
\begin{align*}
\mathcal{J}_{\alpha}(f) 
&\geq \frac{1}{2}\|u\|_{L^2(\Omega)}^2 - \int\limits_{W}(-\D)^s \bar{\varphi} f \,dx
=  \frac{1}{2}\|u\|_{L^2(\Omega)}^2 - \int\limits_{\Omega} \overline{u} u  \,dx\\
&\geq \frac{1}{2}\|u\|_{L^2(\Omega)}^2 - \|\overline{u}\|_{L^2(\Omega)}\|u\|_{L^2(\Omega)}
\geq \frac{1}{4}\|u\|_{L^2(\Omega)}^2 - \|\overline{u}\|_{L^2(\Omega)}^2.
\end{align*}
Thus in particular $\mathcal{J}_{\alpha}(f) \geq - \|\overline{u}\|_{L^2(\Omega)}^2$ for all $f \in \widetilde{H}^s(W)$. The formula $\mathcal{J}_{\alpha}(f_{\alpha}) = -\frac{1}{2} \| \hat{u}_{\alpha} \|_{L^2(\Omega)}^2$ yields that 
\[
\| \hat{u}_{\alpha} \|_{L^2(\Omega)}^2 \leq 2 \|\overline{u}\|_{L^2(\Omega)}^2, \qquad \alpha > 0.
\]
The Vishik-Eskin regularity estimates, see \cite[Theorem 3.1]{Grubb} (here we use that $\Omega$ has $C^{\infty}$ boundary), imply that for some $\beta > 0$ 
\[
\| \hat{\varphi}_{\alpha} \|_{\widetilde{H}^{s+\beta}(\Omega)} \leq C, \qquad \alpha > 0.
\]
Compact Sobolev embedding implies that, for some sequence $\alpha \to 0$, 
\[
\hat{\varphi}_{\alpha} \to \psi \text{ in $\widetilde{H}^s(\Omega)$}.
\]
The convergence $(-\Delta)^s \hat{\varphi}_{\alpha}|_{W} \to h = (-\Delta)^s \overline{\varphi}|_W$ in $H^{-s}(W)$ implies that 
\[
(-\Delta)^s \psi|_W = (-\Delta)^s \overline{\varphi}|_W.
\]
Since also $\psi|_W = \overline{\varphi}|_W = 0$, weak unique continuation for the fractional Laplacian implies that $\psi = \overline{\varphi}$ in $\R^n$. This concludes the proof.
\end{proof}

\section{Recovery of $q$}
\label{sec:recov}

In this section we present the argument for Theorem \ref{thm_main}, taking the results of Theorems \ref{thm_ucp_constructive} and \ref{thm_ucp_positive_measure} for granted. The main issue here is to rule out that $u$ vanishes on a too large subset of $\Omega$ in order to define $q$ by means of the quotient $\frac{(-\D)^s u}{u}$. The control on the size of the nodal set of $u$ is ensured by the measurable unique continuation property of Theorem \ref{thm_ucp_positive_measure}.

\begin{proof}[Proof of Theorem \ref{thm_main}]
\emph{Step 1: Recovery of $u$.}
By assumption, for some known $f\in \widetilde{H}^{s}(W_1) \setminus \{ 0 \}$, we are given $\Lambda_q f|_{W_2} = (-\D)^s u|_{W_2}$. Then, the function $v:=u-f$ satisfies
\begin{align*}
((-\D)^s + q) v & = -(-\D)^s f \mbox{ in } \Omega,\\
v & = 0 \mbox{ in } \Omega_e.
\end{align*}
In particular, $v \in \widetilde{H}^{s}(\Omega)$. Hence, by Theorem \ref{thm_ucp_constructive} (or any of the other reconstruction schemes presented in Section \ref{sec:Aq}) $v$ can be reconstructed from the knowledge $(-\Delta)^s v|_{W_2}$. But linearity and the definition of $\Lambda_q$ yield
\begin{align*}
(-\D)^s v|_{W_2} = (-\D)^s u|_{W_2} - (-\D)^s f|_{W_2}
= \Lambda_q f|_{W_2} - (-\D)^s f|_{W_2}. 
\end{align*}
Since $f \in \widetilde{H}^{s}(W_1)$ is assumed to be known, we can constructively recover $v$ from $\Lambda_q f|_{W_2}$. As $u=f + v$, this also yields the constructive recovery of the full function $u$ in $\R^n$.\\

\emph{Step 2: Reconstruction of the potential $q$.}
We split the reconstruction argument for $q$ into two steps and first deal with $q\in C^0(\overline{\Omega})$ and then with $q \in L^{\infty}(\Omega)$.\\

\emph{Step 2a: $q\in C^0(\overline{\Omega})$.} We note that by the fractional Schr\"odinger equation \eqref{eq:frac_Schr}, which is obeyed by $u$, we have
\begin{align*}
q(x) = \frac{(-\D)^s u(x)}{u(x)} 
\end{align*}
for almost every $x \in \Omega$ such that $u(x)\neq 0$. We claim that this suffices to recover $q$ in the whole of $\Omega$ by invoking the weak unique continuation property of the fractional Laplacian and the continuity of $q$. Indeed, fix an arbitrary point $x_0 \in \Omega$. Then the weak unique continuation principle for the fractional Laplacian implies that there exists a sequence $(x_k)$ with $\Omega \ni x_k \rightarrow x_0\in \Omega$ and $u(x_k)\neq 0$. Indeed, else $u=0$ on an open subset of $\Omega$, but by the weak unique continuation property this would entail that $u\equiv 0$, which is impossible since $f$ is not identically zero. Hence, by continuity,
\begin{align*}
q(x_0) = \lim\limits_{k\rightarrow \infty} q(x_k) = \lim\limits_{k\rightarrow \infty} \frac{(-\D)^s u(x_k)}{u(x_k)}.
\end{align*}
Since $x_0 \in \Omega$ was arbitrary, this concludes the argument for the recovery of continuous potentials.\\

\emph{Step 2b: $q\in L^{\infty}(\Omega)$, $s\geq \frac{1}{4}$.}
Since $q \in L^{\infty}(\Omega)$, the potential $q$ is only defined up to a null set. By the measurable boundary unique continuation result of Theorem \ref{thm_ucp_positive_measure}, there exists no set $E \subset \Omega$ with $|E|>0$ such that $u|_{E} \equiv 0$. Hence, the quotient
\begin{align*}
q(x)= \frac{(-\D)^s u(x)}{ u(x)}
\end{align*}
is well defined for almost every $x\in \Omega$, which thus allows us to recover $q\in L^{\infty}(\Omega)$.
\end{proof}

\section{Unique Continuation from Measurable Sets}
\label{sec:ucp}

In the sequel, we seek to prove the following unique continuation result from measurable sets:

\begin{prop}[Measurable UCP]
\label{prop:mUCP}
Let $\Omega \subset \R^n$ with $n\geq 1$ be a bounded open set and let $q \in L^{\infty}(\Omega)$. Let $s\in [\frac{1}{4},1)$ and assume that $u \in H^s(\R^n)$ satisfies 
\begin{align}
\label{eq:schr}
((-\D)^s + q)u=0 \mbox{ in } \Omega. 
\end{align}
If for some measurable set $E \subset \Omega$ with $|E|>0$ we have $u|_{E}=0$, then $u \equiv 0$ in $\R^n$.
\end{prop}

In order to prove this result, we rely on unique continuation arguments for local equations. To this end,
we recall that the nonlocal Schr\"odinger equation \eqref{eq:schr} can also be ``localized" by means of the Caffarelli-Silvestre extension. More precisely, for any $U \subset \R^{n+1}_+$ we set
\begin{align*}
H^{1}(U, x_{n+1}^{1-2s}):= \{v \in \mathcal{D}'(U): \|x_{n+1}^{\frac{1-2s}{2}} v\|_{L^2(U)} + \|x_{n+1}^{\frac{1-2s}{2}}\nabla v\|_{L^2(U)}<\infty \}.
\end{align*}
Phrased in this notation, the article \cite{CS07} shows that
for any $u\in H^{s}(\R^n)$, the unique solution $\tilde{u}\in H^{1}(\R^{n+1}_+, x_{n+1}^{1-2s})$ of
\begin{align*}
\nabla \cdot x_{n+1}^{1-2s} \nabla \tilde{u} & = 0 \mbox{ in } \R^{n+1}_+,\\
\tilde{u} & = u \mbox{ on } \R^n \times \{0\},
\end{align*}
satisfies $c_{n,s}\lim\limits_{x_{n+1}\rightarrow 0} x_{n+1}^{1-2s}\p_{n+1} \tilde{u} = (-\D)^s u$ (as a limit in $H^{-s}(\R^n)$) for some constant $c_{n,s} \neq 0$. 
Hence, \eqref{eq:schr} can be viewed as the following Neumann (or Robin) problem: 
\begin{align}
\label{eq:Neumann}
\begin{split}
\nabla \cdot x_{n+1}^{1-2s} \nabla \tilde{u} & = 0 \mbox{ in } \R^{n+1}_+,\\
c_{n,s} \lim\limits_{x_{n+1} \rightarrow 0} x_{n+1}^{1-2s} \p_{n+1} \tilde{u} & =  -q \tilde{u} \mbox{ on } \Omega \times \{0\}.
\end{split}
\end{align}
Proposition \ref{prop:mUCP} will follow if we can show that any solution $\tilde{u}$, whose Dirichlet data vanishes in a set of positive measure and whose Robin data vanishes on an open subset of the boundary, must be identically zero. This is close to the boundary unique continuation results for the standard Laplacian, see e.g.\ \cite{AdolfssonEscauriaza, TaoZhang}, which correspond to the case $s=1/2$ for Dirichlet or Neumann data (but not Robin data). As in these works, we will base our proof on certain boundary doubling estimates for the solution $\tilde{u}$.

With slight abuse of notation, in the sequel, we will not distinguish between $\tilde{u}$ and $u$ and will use the same symbol both for the Caffarelli-Silvestre extension and for its boundary values.

\subsection{Auxiliary results}
\label{sec:aux_1}

We recall several auxiliary results which will be needed in the proof of Proposition \ref{prop:mUCP}. Most of these (or slight variations of these) can be found in \cite{Rue15} and \cite{RS17a}. If $x_0 \in \R^n$, we will identify $x_0$ with $(x_0, 0) \in \R^{n+1}$ and use the notation 
\[
B_r^+(x_0) = \{ x \in \R^{n+1}_+ \,:\, |x-x_0| < r \}, \qquad B_r'(x_0) = \{ x' \in \R^n \,:\, |x'-x_0| < r \}.
\]
If $x_0=0$ we will just write $B_r^+$ and $B_r'$.

We first recall Caccioppoli's inequality for the Caffarelli-Silvestre extension.

\begin{lem}[Caccioppoli]
\label{lem:Cacc}
Let $s\in (0,1)$ and $r > 0$. Let $u\in H^{1}(B_{4r}^+, x_{n+1}^{1-2s})$ be a solution of
\begin{align*}
\nabla \cdot x_{n+1}^{1-2s} \nabla u & = 0 \mbox{ in } B_{4r}^+.
\end{align*}
Assume that $u, \lim\limits_{x_{n+1}\rightarrow 0} x_{n+1}^{1-2s}\p_{n+1} u \in L^2(B_{4r}')$.
Then there exists $C = C_{n,s} > 0$ such that 
\begin{align*}
\|x_{n+1}^{\frac{1-2s}{2}} \nabla u\|_{L^2(B_r^+)}
\leq C r^{-1} \|x_{n+1}^{\frac{1-2s}{2}} u\|_{L^2(B_{2r}^+)} + C \|u\|_{L^2(B_{2r}')}^{1/2} \|\lim\limits_{x_{n+1}\rightarrow 0} x_{n+1}^{1-2s} \p_{n+1} u\|_{L^2(B_{2r}')}^{1/2}.
\end{align*}
\end{lem}

\begin{proof}
The proof follows as in Lemma 4.5 in \cite{RS17a}. However instead of dealing with the boundary contributions by duality, we directly use the Cauchy-Schwarz inequality for them (in conjunction with the assumption that the boundary values and the weighted normal derivative at the boundary are $L^2$ functions).
\end{proof}

Next we recall some trace and Sobolev estimates for functions in weighted $H^1$ spaces.

\begin{lem}[Trace estimates]
\label{lem:trace}
Let $s\in (0,1)$ and let $r > 0$. Let $u\in H^{1}(B_{4r}^+, x_{n+1}^{1-2s})$. Then $u|_{B_{r}'} \in H^s(B_r')$, and there is $C = C_{n,s} > 0$ such that 
\begin{align*}
\| u\|_{L^2(B_{r}')}
\leq C (r^{s} \|x_{n+1}^{\frac{1-2s}{2}} \nabla u\|_{L^2(B_{2r}^+)} 
+ r^{s-1} \|x_{n+1}^{\frac{1-2s}{2}} u\|_{L^2(B_{2r}^+)}).
\end{align*}
Moreover, if $n\geq 2$, or $n=1$ and $s\in (0,1/2)$, there is $C = C_{n,s} > 0$ such that 
\[
\|u\|_{L^{2^*(s)}(B_r')}
\leq C(
\|x_{n+1}^{\frac{1-2s}{2}} \nabla u\|_{L^2(B_{2r}^+)}
+r^{-1} \|x_{n+1}^{\frac{1-2s}{2}} u\|_{L^2(B_{2r}^+)}),
\]
where $2^{\ast}(s)= \frac{2n}{n-2s} \in (1,\infty)$.
\end{lem}
\begin{proof}
The estimates are scaling invariant, and thus it is enough to prove them when $r=1$. Let $\eta$ be a cut-off function which is supported in $B_2^+$ and is equal to one near $\overline{B}_1^+$. Then $\eta u \in H^1(\R^{n+1}_+, x_{n+1}^{1-2s})$, and by the trace theorem for this space (see e.g.\ \cite[Lemma 4.4]{RS17a}) one has $\eta u|_{\R^n} \in H^s(\R^n)$ and 
\[
\|u\|_{L^2(B_1')} \leq \|u\|_{H^s(B_1')} \leq C_{n,s} (\|x_{n+1}^{\frac{1-2s}{2}} \nabla u\|_{L^2(B_2^+)} 
+ \|x_{n+1}^{\frac{1-2s}{2}} u\|_{L^2(B_{2}^+)}).
\]
The Sobolev embedding also yields $\|u\|_{L^{2^*(s)}(B_1')} \leq C_{n,s} \|u\|_{H^s(B_1')}$ unless $n=1$ and $s \in [1/2,1)$. The result follows.
\end{proof}

Next we recall a (slight extension of a) result from \cite{Rue15}. The constants from this point on will be denoted by $M$ and they will in general depend on the solution $u$.

\begin{lem}[Doubling]
\label{lem:doubling}
Let $\Omega \subset \R^n$ be a bounded open set and let $q \in L^{\infty}(\Omega)$. Let $s\in [\frac{1}{4},1)$ and assume that $u \in H^s(\R^n)$ is a solution to \eqref{eq:schr}. Then there exists $r_0=r_0(\|q\|_{L^{\infty}(\Omega)})>0$ such that for each $x_0 \in \Omega$ there exists a constant $M=M(\dist(x_0, \partial \Omega), n,s,u)>0$ such that for all $r\in (0,r_0\dist(x_0,\partial \Omega)/10)$
\begin{align*}
\|x_{n+1}^{\frac{1-2s}{2}} u\|_{L^2(B_{2r}^+(x_0))}
+ r \|x_{n+1}^{\frac{1-2s}{2}} \nabla u\|_{L^2(B_{2r}^+(x_0))} 
\leq M (\|x_{n+1}^{\frac{1-2s}{2}} u\|_{L^2(B_r^+(x_0))}
+ r \|x_{n+1}^{\frac{1-2s}{2}} \nabla u\|_{L^2(B_r^+(x_0))} ).
\end{align*}
\end{lem}

\begin{proof}
The proof of Lemma \ref{lem:doubling} follows along the same lines as the proof of Proposition 4.1 in \cite{Rue15} with two slight modifications: As the Schr\"odinger equation is only assumed to hold on the bounded domain $\Omega$, we have to choose the cut-off function $\eta$ such that it is supported in $\Omega$. This gives rise to the dependence on the distance to the boundary. Secondly, in the Carleman estimate, we keep the gradient terms in the small balls (instead of estimating them by means of Caccioppoli's inequality). This necessitates a slight upgrade of the Carleman estimate from \cite[Proposition 4.1]{Rue15}, which uses ideas from \cite[Remark 4]{Rue15}. We present a self-contained argument for the upgraded Carleman estimate in the appendix.

For self-containedness, we present the details. Without loss of generality (by scaling and translating we can always achieve this), we may also assume that $B_4' \subset \Omega$ and $x_0=0$.\\

\emph{Step 1: The Carleman estimate.}
We begin by recalling the Carleman estimate from \cite[Proposition 4.1]{Rue15}, in the upgraded form given in Proposition \ref{prop:Carl}, which allows us to treat $L^{\infty}$ potentials. For a solution $w \in H^{1}(B_5^+, x_{n+1}^{1-2s})$ with $\supp(w) \subset \overline{B_{4}^+ \setminus B_{r_1}^+}$ and $r_1 \in (0,1)$ of
\begin{align}
\label{eq:inhom}
\begin{split}
\nabla \cdot x_{n+1}^{1-2s} \nabla w & = f \mbox{ in } B_5^+,\\
\lim\limits_{x_{n+1}\rightarrow 0} x_{n+1}^{1-2s} \p_{n+1} w & = Vw \mbox{ on } B_5',
\end{split}
\end{align}
for parameters $\tau \geq \tau_0 \geq 1$ and for the weight function $\phi(x):= \psi(|x|)$ with
\begin{align*}
\psi(r) = -\ln(r) + \frac{1}{10}\left( \ln(r) \arctan(\ln(r))- \frac{1}{2} \ln(1+ \ln^2(r))  \right),
\end{align*}
we have for any $r_2\in (2 r_1,3)$
\begin{align}
\label{eq:Carl}
\begin{split}
& \tau^{\frac{1}{2}}\ln(r_2/r_1)^{-1}\|e^{\tau \phi} x_{n+1}^{\frac{1-2s}{2}} |x|^{-1} w \|_{L^2(B_{r_2}^+)}
+ \tau^{-\frac{1}{2}}\ln(r_2/r_1)^{-1} \|e^{\tau \phi} x_{n+1}^{\frac{1-2s}{2}} \nabla w\|_{L^2(B_{r_2}^+)}\\
& + \tau^s \|e^{\tau \phi} (1+\ln^2(|x|))^{-1/2} |x|^{-s} w\|_{L^2(B_5')}\\
&+ \tau \|e^{\tau \phi} (1+\ln^2(|x|))^{-1/2} x_{n+1}^{\frac{1-2s}{2}}|x|^{-1} w \|_{L^2(B_5^+)}
+  \|e^{\tau \phi}  (1+\ln^2(|x|))^{-1/2} x_{n+1}^{\frac{1-2s}{2}} \nabla w \|_{L^2(B_5^+)}\\
&\leq C \tau^{-\frac{1}{2}} \| e^{\tau \phi} |x| x_{n+1}^{\frac{2s-1}{2}} f\|_{L^2(B_5^+)} + \tau^{\frac{1-2s}{2}} \|e^{\tau \phi} |x|^{s} V w\|_{L^2(B_5')}.
\end{split}
\end{align}
This Carleman estimate is obtained by using the ideas which are explained in \cite[Remark 4]{Rue15}. These were already used in \cite[Proposition 4.1]{Rue15} and \cite[Remark 3.8, Corollary 3.11 and Proposition 4.9]{KRS16a}, in order to derive doubling inequalities. While in the setting of \cite[Proposition 4.1]{Rue15} and \cite[Proposition 4.9]{KRS16a} we only needed this on the level of $w$, we here also need it on the level of the gradient $\nabla w$. We give a new, self-contained proof of the Carleman estimate in the appendix which is similar to the splitting arguments that have been used in \cite{KLW16} and \cite{KRS16b}.\\

\emph{Step 2: Application of the Carleman inequality.}
In order to turn our solution $u$ from \eqref{eq:Neumann} (recall that we write $u$ instead of $\tilde{u}$) into a form in which the estimate \eqref{eq:Carl} can be applied, we multiply $u$ with a radial cut-off function $\eta\geq 0$ with $\supp(\eta) \subset \overline{B_{4}^+ \setminus B_{r}^+}$, $\eta = 1$ in $B_{3}^+ \setminus B_{2r}^+$, and $|\nabla \eta| \leq \frac{C}{r}$ in $B_{2r}^+ \setminus B_r^+$. Here $r$ is any number with $0 < r < 1$, and we will track the dependence of the constants on $r$. We note that the function $w:= \eta u$ satisfies \eqref{eq:inhom} with
\begin{align}
\label{eq:inhom_f}
f(x) = 2 x_{n+1}^{1-2s} \nabla \eta \cdot \nabla u + u \nabla \cdot x_{n+1}^{1-2s} \nabla \eta,
\end{align}
and that by the radial form of $\eta$ we in particular have
\begin{align*}
c_{n,s}\lim\limits_{x_{n+1}\rightarrow 0} x_{n+1}^{1-2s} \p_{n+1} w = -q w \mbox{ on } B_5',
\end{align*}
i.e.\ we do not catch $\eta$ derivatives in the normal derivative, and $V$ in \eqref{eq:inhom} is given by $V = -\frac{1}{c_{n,s}} q$.

We may hence apply the Carleman estimate \eqref{eq:Carl} to $w = \eta u$. In the case $s > 1/4$, we can absorb the boundary term on the right hand side into the left hand side if $\tau \geq \tau_0$ with $\tau_0$ sufficiently large. The same can be done when $s = 1/4$ by using the $(1+\ln^2(|x|))^{-1/2}  |x|^{-s}$ and $|x|^{s}$ weights in the boundary terms, provided that we replace $\eta$ by $\eta(\,\cdot\,/r_0)$ where $r_0$ is a small constant depending on $\|q\|_{L^{\infty}}$. We assume that this has been done, and we can thus drop all boundary contributions.
This turns \eqref{eq:Carl} into  
\begin{align*}
\begin{split}
& \tau^{\frac{1}{2}} \ln(r_2/r_1)^{-1} \|e^{\tau \phi} x_{n+1}^{\frac{1-2s}{2}} |x|^{-1} w\|_{L^2(B_{r_2}^+)} + \tau^{-\frac{1}{2}} \ln(r_2/r_1)^{-1}  \|e^{\tau \phi} x_{n+1}^{\frac{1-2s}{2}} \nabla w \|_{L^2(B_{r_2}^+)}\\
& + \tau \|e^{\tau \phi} (1+\ln^2(|x|))^{-1/2} |x|^{-1} x_{n+1}^{\frac{1-2s}{2}} w \|_{L^2(B_4^+)}
+  \|e^{\tau \phi}  (1+\ln^2(|x|))^{-1/2} x_{n+1}^{\frac{1-2s}{2}} \nabla w \|_{L^2(B_4^+)}\\
&\leq C \tau^{-\frac{1}{2}} \|e^{\tau \phi}|x| x_{n+1}^{\frac{2s-1}{2}} f\|_{L^2(B_4^+)}.
\end{split}
\end{align*}
Next, plugging in the form of $f$ from \eqref{eq:inhom_f}, choosing $r_1=r$, $r_2=4r$ and using the support assumption of $\eta = \eta(|x|)$ (which in particular implies that $|\nabla \cdot x_{n+1}^{1-2s} \nabla \eta| \leq C_s x_{n+1}^{1-2s}$ in $B_{4}^+ \setminus B_{3}^+$ and $|\nabla \cdot x_{n+1}^{1-2s} \nabla \eta| \leq C_s x_{n+1}^{1-2s} r^{-2}$ in $B_{2r}^+ \setminus B_r^+$), we infer that 
\begin{align*}
\begin{split}
& \tau^{\frac{1}{2}}  \|e^{\tau \phi} x_{n+1}^{\frac{1-2s}{2}} |x|^{-1} u\|_{L^2(B_{4 r}^+\setminus B_{2r}^+)} 
+ \tau^{-\frac{1}{2}} \|e^{\tau \phi} x_{n+1}^{\frac{1-2s}{2}} \nabla u\|_{L^2(B_{4 r}^+\setminus B_{2r}^+)}\\
& + \tau \|e^{\tau \phi} (1+\ln^2(|x|))^{-1/2} x_{n+1}^{\frac{1-2s}{2}} |x|^{-1} u \|_{L^2(B_{\frac{5}{2}}^+\setminus B_2^+)}
+  \|e^{\tau \phi} (1+\ln^2(|x|))^{-1/2} x_{n+1}^{\frac{1-2s}{2}} \nabla u \|_{L^2(B_{\frac{5}{2}}^+ \setminus B_2^+)}\\
&\leq 4 C \tau^{-\frac{1}{2}} \| e^{\tau \phi} x_{n+1}^{\frac{1-2s}{2} } \nabla u \|_{L^2(B_4^+\setminus B_3^+)}
+ 4C \tau^{-\frac{1}{2}} \| e^{\tau \phi} x_{n+1}^{\frac{1-2s}{2}} u\|_{L^2(B_4^+ \setminus B_3^+)}\\
& \quad + 4C \tau^{-\frac{1}{2}} \|e^{\tau \phi} x_{n+1}^{\frac{1-2s}{2}} \nabla u \|_{L^2(B_{2r}^+ \setminus B_{r}^+)}
+ 4C \tau^{-\frac{1}{2}} r^{-1} \|e^{\tau \phi} x_{n+1}^{\frac{1-2s}{2}} u\|_{L^2(B_{2r}^+ \setminus B_{r}^+)}.
\end{split}
\end{align*}
Using the domain structure, the monotonicity properties of $\psi$ and the trivial estimate $\tau\geq 1$, 
this can be further estimated by 
\begin{align} \label{doubling_intermediate_estimate}
\begin{split}
&  \tau^{-\frac{1}{2}} e^{\tau \psi(4r)} r^{-1} \| u \|_{H^1_r(B_{4r}^+\setminus B_{2r}^+,x_{n+1}^{1-2s})}
 + e^{\tau \psi(\frac{5}{2})} \| u \|_{H^1(B_{\frac{5}{2}}^+\setminus B_2^+,x_{n+1}^{1-2s})}\\
&\leq 4 C \tau^{-\frac{1}{2}} e^{\tau \psi(3)}\| u \|_{H^1(B_4^+\setminus B_3^+, x_{n+1}^{1-2s})}
+ 4 C \tau^{-\frac{1}{2}} e^{\tau \psi(r)} r^{-1} \| u \|_{H^1_r(B_{2r}^+ \setminus B_{r}^+, x_{n+1}^{1-2s})}.
\end{split}
\end{align}
Here on $B_{4r}^+ \setminus B_{2r}^+$ (and similarly on $B_{2r}^+ \setminus B_{r}^+$) we have used the notation
\begin{align*}
\| u \|_{H^1_r(B_{4r}^+\setminus B_{2r}^+,x_{n+1}^{1-2s})}
:= \|x_{n+1}^{\frac{1-2s}{2}} u\|_{L^2(B_{4r}^+ \setminus B_{2r}^+)}
+ r\|x_{n+1}^{\frac{1-2s}{2}} \nabla u\|_{L^2(B_{4r}^+ \setminus B_{2r}^+)},
\end{align*}
in order to avoid always having to spell out the full norms.

Adding the term $\tau^{-\frac{1}{2}}e^{\tau \psi(4r)} r^{-1} \| u \|_{H^1_r(B_{2r}^+,x_{n+1}^{1-2s})}$ to both sides of \eqref{doubling_intermediate_estimate}, and using the fact that $e^{\tau \psi(4r)} \leq e^{\tau\psi(r)}$ for $\tau \geq 0$,
we further obtain
\begin{align}
\label{eq:Carl_2}
\begin{split}
&  \tau^{-\frac{1}{2}}e^{\tau \psi(4r)} r^{-1} \| u \|_{H^1_r(B_{4r}^+,x_{n+1}^{1-2s})}
 + e^{\tau \psi(\frac{5}{2})} \| u \|_{H^1(B_{\frac{5}{2}}^+\setminus B_2^+,x_{n+1}^{1-2s})}\\
&\leq C \tau^{-\frac{1}{2}} e^{\tau \psi(3)}\| u \|_{H^1(B_4^+, x_{n+1}^{1-2s})}
+ C \tau^{-\frac{1}{2}} e^{\tau \psi(r)} r^{-1} \|u \|_{H^1_r(B_{2r}^+ , x_{n+1}^{1-2s})}.
\end{split}
\end{align}
Next, we choose $\tau_0 \geq 1$, depending on the fixed function $u$, so large that for $\tau \geq \tau_0$ one has 
\begin{align*}
C \tau^{-\frac{1}{2}} e^{\tau \psi(3)}\| u \|_{H^1(B_4^+, x_{n+1}^{1-2s})} 
\leq \frac{1}{2} 
e^{\tau \psi(\frac{5}{2})} \| u \|_{H^1(B_{\frac{5}{2}}^+\setminus B_2^+,x_{n+1}^{1-2s})}.
\end{align*}
In fact, it is enough to arrange that for $\tau \geq \tau_0$, 
\[
e^{\tau(\psi(\frac{5}{2}) - \psi(3))} \tau^{1/2} \geq 2C \frac{\| u \|_{H^1(B_4^+, x_{n+1}^{1-2s})}}{\| u \|_{H^1(B_{\frac{5}{2}}^+\setminus B_2^+,x_{n+1}^{1-2s})}}.
\]
If $u$ is nontrivial, the denominator is nonzero by unique continuation for uniformly elliptic equations (the equation is uniformly elliptic in the interior $\{x_{n+1}>0\}$). Thus by the monotonicity properties of $\psi$ it is possible to find such a $\tau_0$ (which will depend on $u$).
This choice of $\tau_0>0$ then allows us to absorb the first term on the right hand side of \eqref{eq:Carl_2} into the left hand side. As a consequence, we obtain that for all $r\in (0,r_0/10)$ (where $r_0>0$ was the radius obtained in the discussion of the boundary terms above) and for all $\tau \geq \tau_0>1$
\begin{align}
\label{eq:Carl_2a}
\begin{split}
&  \| u \|_{H^1_r(B_{4r}^+,x_{n+1}^{1-2s})}
\leq C e^{\tau (\psi(r)- \psi(4r))}\|u \|_{H^1_r(B_{2r}^+, x_{n+1}^{1-2s})}.
\end{split}
\end{align}

Finally, we note that by the choice of $\psi(t)$ the difference $\psi(r)- \psi(4r)$ can be bounded from below and above independently of $r>0$. Indeed, we have
\begin{align*}
\psi(r)-\psi(4r) = \ln(4) + \frac{1}{10} \left(\ln(r) \arctan(\ln(r))-\ln(4r) \arctan(\ln(4r)) \right) - \frac{1}{20}\ln\left( \frac{1+\ln^2(r)}{1+\ln^2(4r)}\right) .
\end{align*}
Since
\begin{align*}
\frac{1+\ln^2(r)}{1+\ln^2(4r)} \rightarrow 1 \mbox{ as } r \rightarrow 0,
\end{align*}
it suffices to estimate the difference 
\begin{align*}
\ln(r) \arctan(\ln(r))-\ln(4r) \arctan(\ln(4r)) .
\end{align*}
For this we recall the Taylor expansion of $\arctan$ at infinity
\begin{align*}
\arctan(t) = \frac{\pi}{2}-\frac{1}{t} + o(t^{-2}), \ t \gg 1.
\end{align*}
This yields
\begin{align*}
\left| \ln(r) \arctan(\ln(r))-\ln(4r) \arctan(\ln(4r))  \right| 
\leq  \frac{\pi}{2}\ln(4) + o\left(\frac{1}{\ln(r)} \right).
\end{align*}
As a consequence, if $r\in (0,r_0)$ with $r_0>0$ sufficiently small, this then implies
\begin{align} \label{psi_difference}
\ln(4) - \frac{2}{5} \leq |\psi(r)-\psi(4r)| \leq \ln(4),
\end{align}
which concludes the proof of the desired doubling estimate.
\end{proof}

\begin{rmk}
\label{rmk:Spec_gap}
As noted in Section 7.3 and in Remark 3 in \cite{Rue15} the Carleman estimate can be strengthened in the presence of a spectral gap. Due to (for instance) the results in \cite{KRS16}, Section 8.3, it is now known that a spectral gap analogous to the one for uniformly elliptic equations also holds in the setting of the fractional Laplacian. In particular, this implies that the strengthened bounds from Remark 3 and Section 7.3 in \cite{Rue15} hold and that in the Carleman estimate \eqref{eq:Carl} it is also possible to consider logarithmic weight functions (without additional convexifications).
\end{rmk}

\begin{rmk}
\label{rmk:s_small}
In the case $s\in (0,1/4)$ the Carleman estimate \eqref{eq:Carl} can no longer be used to derive doubling estimates, as it is no longer possible to absorb the boundary term from the right hand side of \eqref{eq:Carl} into the left hand side (due to the mismatch in the powers of $\tau$). Due to the subelliptic nature of the Carleman estimate this loss in $\tau$ seems unavoidable (for weights which are purely tangential). However, a similar Carleman estimate can be proved, if more regularity in $q$ is required (in \cite{Rue15} the author assumed $C^1$ regularity, but only $C^1$ regularity in the radial directions was used in the corresponding argument). Using this estimate, it would have been possible to derive an analogous doubling inequality and to extend our measurable unique continuation argument to $s\in (0,1/4)$ if $q \in C^1$. However, as the unique continuation property from measurable sets with $C^1$ potentials was already proved in \cite{FF14}, and as our reconstruction argument can directly deal with continuous potentials $q$ by invoking the weak unique continuation property, we do not further pursue this here.
\end{rmk}

Combining the estimates from Lemmas \ref{lem:Cacc}-\ref{lem:doubling}, we obtain the following estimate.

\begin{lem}[Gradient estimate]
\label{cor:comp}
Let $\Omega \subset \R^n$ be a bounded open set and let $q \in L^{\infty}(\Omega)$. Let $s\in [\frac{1}{4},1)$ and assume that $u \in H^s(\R^n)$ satisfies \eqref{eq:schr}. Then for each $x_0\in \Omega$ there exists a constant $M=M(\dist(x_0,\p \Omega),n,s, \|q\|_{L^{\infty}(\Omega)}, u)>0$ such that for each $r$ with 
\[
0 < r < \min\{\dist(x_0, \partial \Omega)/2, (2 M \|q\|_{L^{\infty}(\Omega)}^{1/2})^{-\frac{1}{s}} \})
\]
one has 
\begin{align*}
\|x_{n+1}^{\frac{1-2s}{2}} \nabla u \|_{L^2(B_{r}^+(x_0))}
\leq M r^{-1} \|x_{n+1}^{\frac{1-2s}{2}} u\|_{L^2(B_{r}^+(x_0))}.
\end{align*}
\end{lem}

\begin{proof}
We observe that by the doubling estimate from Lemma \ref{lem:doubling} we have
\begin{align}
\label{eq:grad_1}
\|x_{n+1}^{\frac{1-2s}{2}} \nabla u\|_{L^2(B_{r}^+)}
\leq M(\|x_{n+1}^{\frac{1-2s}{2}} \nabla u\|_{L^2(B_{r/4}^+)} +
r^{-1} \|x_{n+1}^{\frac{1-2s}{2}}  u\|_{L^2(B_{r/4}^+)}).
\end{align}
By Caccioppoli's inequality (Lemma \ref{lem:Cacc}), where we use that on $B_{r/2}'$ 
\begin{align*}
c_{n,s}\lim\limits_{x_{n+1}\rightarrow 0} x_{n+1}^{1-2s} \p_{n+1} u = (-\Delta)^s u = -q u,
\end{align*} 
and by Lemma \ref{lem:trace} we further obtain
\begin{align}
\label{eq:grad_2}
\begin{split}
&\|x_{n+1}^{\frac{1-2s}{2}} \nabla u\|_{L^2(B_{r/4}^+)}
\leq C_{n,s}( r^{-1} \|x_{n+1}^{\frac{1-2s}{2}}  u\|_{L^2(B_{r/2}^+)} + \|q\|_{L^{\infty}(B_{r/2}')}^{1/2} \|u\|_{L^2(B_{r/2}')}) \\
&\leq C_{n,s}( (1+ \|q\|_{L^{\infty}(\Omega)}^{1/2} r^{s}) r^{-1} \|x_{n+1}^{\frac{1-2s}{2}}  u\|_{L^2(B_{r}^+)} + r^{s}\|q\|_{L^{\infty}(\Omega)}^{1/2} \|x_{n+1}^{\frac{1-2s}{2}} \nabla u\|_{L^2(B_{r}^+)}). 
\end{split}
\end{align}
Combining \eqref{eq:grad_1} with \eqref{eq:grad_2} and 
using that $r\leq \left(\frac{1}{2 C_{n,s} M \|q\|_{L^{\infty}(\Omega)}^{1/2}}\right)^{\frac{1}{s}}$ (which allows us to absorb the gradient term on the right hand side into the left hand side) implies
\begin{align*}
\|x_{n+1}^{\frac{1-2s}{2}} \nabla u\|_{L^2(B_{r}^+(x_0))} \leq M r^{-1}\|x_{n+1}^{\frac{1-2s}{2}}u\|_{L^{2}(B_{r}^+(x_0))}.
\end{align*}
This concludes the argument.
\end{proof}

\subsection{Proof of Proposition \ref{prop:mUCP}}
\label{sec:proof_1}

With the auxiliary results of Section \ref{sec:aux_1} at hand, we address the proof of Proposition \ref{prop:mUCP}.

\begin{proof}[Proof of Proposition \ref{prop:mUCP}]
We seek to reduce the claim of Proposition \ref{prop:mUCP} to the boundary unique continuation property for the Caffarelli-Silvestre extension. To this end, we argue by contradiction and assume that $u \in H^s(\R^n)$ is a fixed function that satisfies \eqref{eq:schr}, vanishes on a set $E$ of positive measure, but $u$ is not identically zero. We split the argument into two steps. First we derive a smallness condition at points of density one of $E \subset \Omega$. Then, using the assumption that $u \not\equiv 0$, we show that this smallness property together with a blow-up argument implies a contradiction. We stress that the constants below may in general depend on the fixed solution $u$. \\

\emph{Step 1: Smallness.}
We claim that for any fixed point $x_0 \in E \cap \Omega$ of density one of $E$ and for any $\epsilon>0$, there exists a radius $r_0 > 0$, depending on $x_0,\epsilon,\dist(x_0,\p \Omega),n,s, \|q\|_{L^{\infty}(\Omega)}$, and $u$, such that for all $r\in (0,r_0)$ it holds
\begin{align}
\label{eq:small}
\|u\|_{L^2(B_r'(x_0))} \leq \epsilon r^{s-1}\|x_{n+1}^{\frac{1-2s}{2}} u\|_{L^2(B_r^+(x_0))}.
\end{align}
In order to observe this, we first note that as $x_0$ is a point of density one of $E$, we have that for each $\delta>0$ there exists $r_{\delta}>0$ such that for all $r\in (0,r_{\delta})$
\begin{align}
\label{eq:density_p}
|B_{r}'(x_0)\cap E^{c}| \leq \delta |B_r'(x_0)|.
\end{align}
Next we distinguish two cases. We first assume that either $n\geq 2$, or $n=1$ and $s\in(0,1/2)$; the case $n=1$ and $s\in [1/2,1)$ will be treated in Step 1b below.\\

\emph{Step 1a: The case $n\geq 2$, or $n=1$ and $s\in(0,1/2)$.}
Fixing $r \in (0,r_{\delta})$ (for a value of $\delta$, which is still to be determined) and using that $u=0$ on $E$, we obtain by virtue of H\"older's inequality that
\begin{align}
\label{eq:dense}
\|u\|_{L^2(B_{r}'(x_0))} 
= \|u\|_{L^2(B_{r}'(x_0) \cap E^c )}
\leq |B_{r}'(x_0) \cap E^{c}|^{\frac{s}{n}} \|u\|_{L^{2^{\ast}(s)}(B_r'(x_0))},
\end{align}
where $2^{\ast}(s):= \frac{2n}{n-2s} \in (1,\infty)$ (and where we used the assumption $n\geq 2$ or $n=1$ and $s\in(0,1/2)$). 
Applying the localized Sobolev inequality (Lemma \ref{lem:trace}) as well as doubling (Lemma \ref{lem:doubling}) and Lemma  \ref{cor:comp}, we further bound
\begin{align}
\label{eq:Sob}
\begin{split}
\|u\|_{L^{2^{\ast}(s)}(B_r'(x_0))}
&\leq C( \|x_{n+1}^{\frac{1-2s}{2}}\nabla u\|_{L^2(B_{2r}^+(x_0))}
+ r^{-1}\|x_{n+1}^{\frac{1-2s}{2}} u\|_{L^2(B_{2 r}^+(x_0))})\\
&\leq M r^{-1} \|x_{n+1}^{\frac{1-2s}{2}} u\|_{L^2(B_{r}^+(x_0))}.
\end{split}
\end{align}
Combining \eqref{eq:density_p}, \eqref{eq:dense} and \eqref{eq:Sob} consequently yields
\begin{align*}
\|u\|_{L^2(B_{r}'(x_0))} 
\leq M \delta^{\frac{s}{n}} |B_{r}'(x_0)|^{\frac{s}{n}} r^{-1} \|x_{n+1}^{\frac{1-2s}{2}} u\|_{L^2(B_{r}^+(x_0))}
= M \delta^{\frac{s}{n}} r^{s-1} \|x_{n+1}^{\frac{1-2s}{2}} u\|_{L^2(B_{r}^+(x_0))}.
\end{align*}
Choosing $\delta$ such that $M \delta^{\frac{s}{n}} =\epsilon$ hence implies \eqref{eq:small}.\\ 

\emph{Step 1b: The case $n=1$, $s\in [1/2,1)$.} If $n=1$, $s\in [1/2,1)$, we only have to modify the bounds leading to \eqref{eq:dense} and \eqref{eq:Sob}. Setting $s':=s/2 \in [1/4,1/2)$, we obtain
\begin{align}
\label{eq:dense1}
\|u\|_{L^2(B_{r}'(x_0))} 
= \|u\|_{L^2(B_{r}'(x_0) \cap E^c )}
\leq |B_{r}'(x_0) \cap E^{c}|^{\frac{s'}{n}} \|u\|_{L^{2^{\ast}(s')}(B_r'(x_0))}
\end{align}
as a replacement of \eqref{eq:dense}. Similarly, we infer
\begin{align}
\label{eq:Sob1}
\begin{split}
\|u\|_{L^{2^{\ast}(s')}(B_r'(x_0))}
&\leq C( \|x_{n+1}^{\frac{1-2s'}{2}}\nabla u\|_{L^2(B_{2r}^+(x_0))}
+ r^{-1}\|x_{n+1}^{\frac{1-2s'}{2}} u\|_{L^2(B_{2 r}^+(x_0))})\\
&\leq Cr^{s/2}(\|x_{n+1}^{\frac{1-2s}{2}}\nabla u\|_{L^2(B_{2r}^+(x_0))}
+ r^{-1}\|x_{n+1}^{\frac{1-2s}{2}} u\|_{L^2(B_{2 r}^+(x_0))}).
\end{split}
\end{align}
Thus, combining \eqref{eq:dense1} and \eqref{eq:Sob1} with the estimates from Lemmas \ref{lem:doubling} and \ref{cor:comp} then also implies that 
\begin{align*}
\|u\|_{L^2(B_{r}'(x_0))} 
\leq M \delta^{\frac{s}{2n}} |B_{r}'(x_0)|^{\frac{s}{2n}} r^{s/2} r^{-1} \|x_{n+1}^{\frac{1-2s}{2}} u\|_{L^2(B_{r}^+(x_0))}
= M \delta^{\frac{s}{2n}} r^{s-1} \|x_{n+1}^{\frac{1-2s}{2}} u\|_{L^2(B_{r}^+(x_0))}.
\end{align*}
Choosing $\delta$ such that $M\delta^{\frac{s}{2n}}= \epsilon$ then concludes the argument for \eqref{eq:small}.\\

\emph{Step 2: Vanishing of infinite order.}
Let $x_0 \in E \cap \Omega$ be a point of density one of $E$.
We next use \eqref{eq:small} in combination with a blow-up argument, in order to infer that $u$ vanishes identically in $B_1^+(x_0)$. This will give a contradiction to our assumption that $u$ is not identically zero.

For any $\sigma > 0$, define the function $u_{\sigma}$ in $\R^{n+1}_+$ by 
\begin{align*}
u_{\sigma}(x):= \frac{u(x_0 + \sigma x)}{\sigma^{-\frac{n+1}{2}}\sigma^{-\frac{1-2s}{2}} \|x_{n+1}^{\frac{1-2s}{2}} u\|_{L^2(B_{\sigma}^+(x_0))}}.
\end{align*}
This function is well-defined, as the denominator does not vanish for any choice of $\sigma>0$ (as else by unique continuation $u\equiv 0$, which is ruled out by our assumption).
If $\sigma>0$ is sufficiently small (so as to satisfy the conditions in Lemma \ref{cor:comp}), a rescaling of the estimate
\begin{align*}
\|x_{n+1}^{\frac{1-2s}{2}} u\|_{L^2(B_{2\sigma}^+(x_0))} + \sigma\|x_{n+1}^{\frac{1-2s}{2}}\nabla u\|_{L^2(B_{2\sigma}^+(x_0))}
\leq M  \|x_{n+1}^{\frac{1-2s}{2}} u\|_{L^2(B_{\sigma}^+(x_0))}
\end{align*}
which follows from Lemmas \ref{lem:doubling} and \ref{cor:comp} entails that 
\begin{align*}
\|x_{n+1}^{\frac{1-2s}{2}} u_{\sigma}\|_{L^2(B_{2}^+(0))}
+
\|x_{n+1}^{\frac{1-2s}{2}}\nabla u_{\sigma}\|_{L^2(B_{2}^+(0))}
\leq M  \|x_{n+1}^{\frac{1-2s}{2}} u_{\sigma}\|_{L^2(B_{1}^+(0))} = M.
\end{align*}
Here we used the normalization of $u_{\sigma}$ on $B_1^+(0)$ to infer the last bound. As a consequence of this estimate, Rellich's compactness theorem for these weighted spaces (see \cite[Section 2.2]{FF14}) and the normalization $\|x_{n+1}^{\frac{1-2s}{2}}u_{\sigma}\|_{L^2(B_1^+)}=1$, we infer that there exists $u_0 \in H^{1}(B_2^+, x_{n+1}^{1-2s})$ such that, for some subsequence $\sigma_k \to 0$,
\begin{equation}
\label{eq:bulk_conv}
\begin{gathered}
u_{\sigma_k} \rightharpoonup u_0 \text{ weakly in $H^1(B_{2}^+, x_{n+1}^{1-2s})$}, \qquad u_{\sigma_k} \rightarrow u_0 \text{ strongly in } L^2(B_{2}^+, x_{n+1}^{1-2s}), \\\mbox{ and } \|x_{n+1}^{\frac{1-2s}{2}} u_0\|_{L^2(B_{2}^+)} = 1.
\end{gathered}
\end{equation}
The trace map $H^1(B_{2}^+, x_{n+1}^{1-2s}) \to L^2(B_{3/2}')$, $v \mapsto v|_{B_{3/2}'}$ is compact by Lemma \ref{lem:trace} and by the compact embedding $H^s(B_{3/2}') \to L^2(B_{3/2}')$. In particular, after passing to another subsequence, we also infer that 
\begin{align}
\label{eq:bound_conv}
u_{\sigma_k} \rightarrow u_0 \mbox{ in } L^2(B_{3/2}') \mbox{ and } u_0 \in H^{s}(B_{3/2}').
\end{align}

Note that for $\sigma$ small enough, the function $u_{\sigma}$ is a weak solution of
\begin{align}
\label{eq:weak_1}
\nabla \cdot x_{n+1}^{1-2s} \nabla u_{\sigma}(x) &= 0 \mbox{ in } B_{2}^+,\\
c_{n,s}\lim\limits_{x_{n+1}\rightarrow 0} x_{n+1}^{1-2s} \p_{n+1} u_{\sigma}(x) 
&= - \sigma^{2s} q(x_0 + \sigma x) u_{\sigma}(x) \mbox{ on } B_{2}'.
\end{align}
In particular, for all $\varphi \in C^{\infty}(\R^{n+1}_+)$ with $\supp(\varphi) \subset B_{3/2}^+ \cap B_{3/2}'$ we have
\begin{align}
\label{eq:weak_2}
\int\limits_{B_1^+} x_{n+1}^{1-2s} \nabla u_{\sigma}(x) \cdot \nabla \varphi(x) dx
= - \frac{\sigma^{2s}}{c_{n,s}} \int\limits_{B_1'} q(x_0 + \sigma x) u_{\sigma}(x) \varphi(x) dx.
\end{align}
Using the convergences from \eqref{eq:bulk_conv}, \eqref{eq:bound_conv}, we obtain that in the limit $\sigma \rightarrow 0$ the function $u_0$ solves
\begin{align*}
\int\limits_{B_{3/2}^+} x_{n+1}^{1-2s} \nabla u_{0}(x) \cdot \nabla \varphi(x) dx
= 0 \mbox{ for all } \varphi \in C^{\infty}(\R^{n+1}_+) \mbox{ with }\supp(\varphi) \subset B_{3/2}^+ \cap B_{3/2}', 
\end{align*}
i.e. $u_0$ is a weak solution of 
\begin{align}
\label{eq:limit_eq_local}
\begin{split}
\nabla \cdot x_{n+1}^{1-2s} \nabla u_{0}(x) &= 0 \mbox{ in } B_{3/2}^+,\\
c_{n,s}\lim\limits_{x_{n+1}\rightarrow 0} x_{n+1}^{1-2s} \p_{n+1} u_{0}(x) 
&= 0 \mbox{ on } B_{3/2}'.
\end{split}
\end{align}
By interior regularity for solutions to \eqref{eq:limit_eq_local} (see for instance Proposition 8.2 in \cite{KRS16}, where the proof shows that it is possible to relax the $L^{\infty}$ requirement to a weighted $L^2$ requirement; alternative arguments on a Sobolev scale follow as in Section 4 in \cite{RS17a}), we also infer that $u_0$ is a H\"older continuous solution to \eqref{eq:limit_eq_local} to which the estimates from above can be applied.

Due to the bound \eqref{eq:small}, for $\sigma$ sufficiently small we however further have
\begin{align*}
\|u_{\sigma}\|_{L^2(B_{1}'(x_0))} \leq \epsilon \|x_{n+1}^{\frac{1-2s}{2}} u_{\sigma}\|_{L^2(B_1^+(x_0))}.
\end{align*}
Passing to the limit $\sigma_k \rightarrow 0$ in this estimate and using the strong convergences \eqref{eq:bulk_conv}, \eqref{eq:bound_conv} results in 
\begin{align}
\label{eq:eps}
\|u_0\|_{L^2(B_{1}'(x_0))} \leq \epsilon .
\end{align}
A diagonal sequence argument shows that \eqref{eq:eps} holds for a sequence $\epsilon_k \rightarrow 0$. As a consequence, $u_0 = 0$ in $B_{1}'(x_0)$. As $u_0$ also solves \eqref{eq:limit_eq_local} in $B_1^+$, we arrive at $u_0=0$ and $\lim\limits_{x_{n+1}\rightarrow 0} x_{n+1}^{1-2s} \p_{n+1} u_0= 0$ in $B_{1}' \times \{0\}$. By (boundary) weak unique continuation property of the Caffarelli-Silvestre extension \cite[Proposition 2.2]{Rue15}, this however entails that $u_0\equiv 0$ in $B_1^+$, which contradicts $\|x_{n+1}^{\frac{1-2s}{2}} u_0\|_{L^2(B_1^+)}=1$.
We have reached a contradiction to the assumption $u \not\equiv 0$. It follows that $u \equiv 0$, which concludes the proof.
\end{proof}

\section{Stability of the unique continuation result}
\label{sec:stab}

Theorem \ref{thm_ucp_constructive} states that a function $v \in H^{s}(\R^n)$ with $\supp(v) \subset \overline{\Omega}$ can be uniquely and constructively determined from the data $(-\D)^s v|_{W}$. In other words, the map 
\[
T: H^s_{\overline{\Omega}} \to H^{-s}(W), \ \ Tv = (-\Delta)^s v|_W
\]
is injective and $v$ can be reconstructed from $Tv$. In this section we are interested in the stability properties of recovering $v$ from $Tv$. Since $\overline{\Omega} \cap \overline{W} = \emptyset$, the map $T$ is smoothing and hence compact, and inverting $T$ is an ill-posed problem. It is well known that some stability properties can be restored by restricting the unknowns to a compact set (which corresponds to assuming a priori bounds), or by using a weaker norm for the unknowns. Here we will choose the latter approach and measure errors in the unknowns in the $H^{s'}$ norm where $s' < s$.

Note that $T$ also maps $H^{s'}_{\overline{\Omega}}$ injectively to $H^{-s}(W)$ for any $s' < s$ by the smoothing property. Any closed and bounded set $B$ in $H^s_{\overline{\Omega}}$ is compact in $H^{s'}_{\overline{\Omega}}$, and thus $T: B \subset H^{s'}_{\overline{\Omega}} \to H^{-s}(W)$ is a homeomorphism onto its image. The following result, which is very close to the stability results in \cite{RS17a}, shows that the inverse of $T$ has a logarithmic modulus of continuity in this setup.

\begin{prop}[Stability of unique continuation]
\label{prop:stab}
Let $\Omega, W \subset \R^n$ with $n\geq 1$ be bounded Lipschitz domains with $\overline{\Omega} \cap \overline{W} = \emptyset$, let $0 < s < 1$, and let $s' < s$. Then, there exist constants $C, \sigma >0$, which only depend on $\Omega$, $W$, $n$, $s$, $s'$, such that for any $E > 0$ one has 
\begin{equation} \label{ucp_stability_v}
\|v\|_{H^{s'}_{\overline{\Omega}}} \leq C \frac{E}{\log(C \frac{E}{\|(-\D)^s v\|_{H^{-s}(W)}})^{\sigma}} \ \text{ whenever } \ \|v\|_{H^s_{\overline{\Omega}}} \leq E.
\end{equation}
\end{prop}

We note that for any $v \in H^s_{\overline{\Omega}}$, upon choosing $E = \|v\|_{H^s_{\overline{\Omega}}}$, the estimate \eqref{ucp_stability_v} can be written equivalently as 
\begin{align}
\label{eq:stab}
\|v\|_{H^s_{\overline{\Omega}}} \leq C \exp\left( C \left( \frac{\| v\|_{H^s_{\overline{\Omega}}}}{\|v\|_{H^{s'}_{\overline{\Omega}}}} \right)^{\mu} \right) \|(-\Delta)^s v\|_{H^{-s}(W)}
\end{align}
where $\mu = 1/\sigma$. This inequality states that high oscillations in $v$ give rise to logarithmic instabilities in the recovery of $v$ from $(-\Delta)^s v|_W$.

\begin{proof}
This result will follow rather directly from a propagation of smallness result for the Caffarelli-Silvestre extension \cite[Theorem 5.1]{RS17a}. By definition, the Caffarelli-Silvestre extension $\tilde{w}$ of $v$ is the solution of
\begin{align*}
\nabla \cdot x_{n+1}^{1-2s} \nabla \tilde{w} & = 0 \mbox{ in } \R^{n+1}_+,\\
\tilde{w} & = v \mbox{ on } \R^{n} \times \{0\}.
\end{align*}
We recall that $\tilde{w}$ in particular satisfies (even in a strong sense as $\tilde{w}= 0$ in $\Omega_e$, cf.\ Section 4 in \cite{RS17a})
\begin{align} \label{fractional_laplace_v}
(-\D)^s v|_{W} = c_s \lim\limits_{x_{n+1}\rightarrow 0} x_{n+1}^{1-2s} \p_{n+1} \tilde{w}|_{W \times \{0\}}.
\end{align}

It follows from \cite[Lemma 4.2]{RS17a} that 
\begin{equation} \label{extension_cr_bound}
\|x_{n+1}^{\frac{1-2s}{2}} \tilde{w}\|_{L^2(\R^n \times (0,2))} + \|x_{n+1}^{\frac{1-2s}{2}} \nabla \tilde{w}\|_{L^2(\R^{n+1}_+)} \leq C_s \|v\|_{H^s_{\overline{\Omega}}} \leq C_s E.
\end{equation}
The formula \eqref{fractional_laplace_v} also gives 
\[
\| \lim\limits_{x_{n+1}\rightarrow 0} x_{n+1}^{1-2s} \p_{n+1} \tilde{w}|_{W \times \{0\}}\|_{H^{-s}(W)} = C_s \|(-\Delta)^s v\|_{H^{-s}(W)}.
\]
Thus letting $\eta := C_s \|(-\Delta)^s v\|_{H^{-s}(W)}$, replacing $E$ by $C_s E$, and assuming that $\frac{E}{\eta} > 1$ (this is possible by \cite[Remark 5.2]{RS17a}), it follows from \cite[Theorem 5.1]{RS17a} that 
\begin{equation} \label{extension_ltwo_bound}
\|x_{n+1}^{\frac{1-2s}{2}} \tilde{w}\|_{L^2(2\Omega \times (0,1))} \leq C \frac{E}{\log(C \frac{E}{\eta})^{\sigma}}.
\end{equation}
Here for $c\in (0,5)$, the notation $c \Omega$ denotes the set $\{x\in \R^n: \dist(x,\Omega)\leq c r_0\}$, where $r_0:= \frac{\dist(\Omega, W)}{10}$.
We wish to obtain a similar estimate for $\nabla \tilde{w}$ with $s$ replaced by $s'$. To do this, note that for any $r \in (0,1/2]$ one has 
\begin{align*}
\|x_{n+1}^{\frac{1-2s'}{2}} \nabla \tilde{w}\|_{L^2((3/2) \Omega \times (0,1/2))} &\leq r^{s-s'} \|x_{n+1}^{\frac{1-2s}{2}} \nabla \tilde{w}\|_{L^2((3/2)\Omega \times (0,r))} + \|x_{n+1}^{\frac{1-2s'}{2}} \nabla \tilde{w}\|_{L^2((3/2)\Omega \times (r,1/2))} \\
 &\leq r^{s-s'} E + C r^{-1} \|x_{n+1}^{\frac{1-2s'}{2}} \tilde{w}\|_{L^2(2\Omega \times (0,1))} \\
 &\leq r^{s-s'} E + r^{-1} C \frac{E}{\log(C \frac{E}{\eta})^{\sigma}}.
\end{align*}
In the second line we used \eqref{extension_cr_bound} and a Caccioppoli inequality similar to \cite[Lemma 4.5]{RS17a}, and in the third line we used \eqref{extension_ltwo_bound}. We next choose $r = \min\{r_0, 1/2\}$, where $r_0$ is chosen so that the two terms in the third line above are equal for $r=r_0$. With this choice of $r$ we obtain that 
\begin{equation} \label{extension_nablaw_bound}
\|x_{n+1}^{\frac{1-2s'}{2}} \nabla \tilde{w}\|_{L^2((3/2)\Omega \times (0,1/2))} \leq C \frac{E}{\log(C \frac{E}{\eta})^{\sigma}}
\end{equation}
for some new positive constants $C$ and $\sigma$.

Finally, using the localized trace estimate from \cite[Lemma 4.5]{RS17a} together with \eqref{extension_ltwo_bound} and \eqref{extension_nablaw_bound}, we have 
\[
\|v\|_{H^{s'}_{\overline{\Omega}}} \leq C (\|x_{n+1}^{\frac{1-2s'}{2}} \tilde{w}\|_{L^2((3/2)\Omega \times (0,1/2))} + \|x_{n+1}^{\frac{1-2s'}{2}} \nabla \tilde{w}\|_{L^2((3/2)\Omega \times (0,1/2))}) \leq C \frac{E}{\log(C \frac{E}{\eta})^{\sigma}}
\]
for some $C, \sigma > 0$ that only depend on $n$, $s$, $s'$, $\Omega$, $W$ as required.
\end{proof}

\subsection{Optimality of logarithmic stability}

In order to observe that the logarithmic stability estimate in Proposition \ref{prop:stab} is indeed optimal, i.e.\ to note that an exponential instability is present, we argue similarly as in \cite{RulandSalo_instability}. We claim that the following holds:

\begin{prop}
\label{prop:instab}
Let $s\in (0,1)$, $\Omega =B_1 \subset \R^n$ with $n\geq 1$ and assume that $W= B_{R} \setminus \overline{B}_{R-1}$ for some large constant $R>1$. Then for any $k\in \N$ there exist functions $v_k \in H^{s}(\R^n)$ with $\supp(v_k) \subset \overline{\Omega}$ and $h_k:=(-\D)^s v_k|_{W}$ such that
\begin{align*}
\|v_k \|_{L^2(\Omega)} &=1, \ 
\|h_k\|_{H^{-s}(W)} \leq C e^{-C k}, 
\end{align*}
for some constant $C>0$.
\end{prop}

\begin{proof}
We choose the sequence $\{v_k\}_{k\in \N}$ to be a sequence of $L^2(\Omega)$ normalized eigenfunctions of the Dirichlet Laplacian in $\Omega = B_1$ corresponding to the $k$-th eigenvalue $\lambda_k$ (modding out multiplicities). In particular, this implies that $v_k|_{\partial \Omega}=0$ and denoting, with slight abuse of notation, the zero extension of $v_k$ also by $v_k$, the support condition $\supp(v_k) \subset \overline{\Omega}$ follows. The definition of $v_k$ also yields that $\|v_k\|_{H^1(\R^n)}^2 = 1+\lambda_k \sim k^{2/n}$ by Weyl asymptotics, hence by interpolation $c k^{s/n} \leq \|v_k\|_{H^{s}_{\overline{\Omega}}} \leq C k^{s/n}$ where the constants only depend on $n$ and $s$. In particular, \eqref{eq:stab} for $s' = 0$ then turns into
\begin{align*}
k^{s/n} \leq C\exp(Ck^{\mu s/n})\|h_k\|_{H^{-s}(W)} \mbox{ with } h_k=(-\D)^s v_k|_{W}.
\end{align*} 
In order to show that this is optimal (up to the power of $\mu$) we seek to prove that (for a sequence of $k\rightarrow \infty$)
\begin{align}
\label{eq:inst_aim}
\|h_k\|_{H^{-s}(W)} \leq C \exp(-C k^{\nu}) \mbox{ for some value } \nu >0.
\end{align}

In order to deduce \eqref{eq:inst_aim}, we compute $h_k$: For $x\in W=B_{R} \setminus \overline{B}_{R-1}$ we have
\begin{align}
\label{eq:frac_Lapl_comp}
h_k(x) = (-\D)^s v_k(x)
= - c \int\limits_{B_1} \frac{v_k(y)}{|x-y|^{n+ 2s}} \,dy
= - c \int \limits_{0}^{1} r^{n-1} j_k(r) \int\limits_{S^{n-1}} \frac{H_k(\omega)}{|r' \omega' - r \omega|^{n+ 2s}} \,d\omega \,dr,
\end{align}
where we have introduced polar coordinates $x=r' \omega'$, $y=r \omega$ and have used that solutions to the Dirichlet Laplacian in $\Omega= B_1$ separate variables $v_k(r\omega)= j_k(r)H_{k}(\omega)$, where $j_k$ denotes a generalized Bessel function and where $H_k(\omega)$ is a spherical harmonic of degree $k$ (it does not matter, which spherical harmonic we consider, we hence simply choose any spherical harmonic of degree $k$).

Next we study the kernel 
\begin{align*}
|r' \omega' - r \omega|^{-(n+2s)} 
&= (r')^{-(n+2s)}\left|\omega' - \left(\frac{r}{r'}\right) \omega \right|^{-(n+2s)}\\
& = (r')^{-(n+2s)} \left(1 + \left(\frac{r}{r'}\right)^2 - 2 \frac{r}{r'}(\omega \cdot \omega') \right)^{-\frac{n+ 2s}{2}},
\end{align*}
where by our assumption on $W$ we have that $r' \gg r$. A Taylor expansion of the (one-dimensional) function $(1+t)^{-\frac{n+2s}{2}}$ around $t=0$ then yields
\begin{align*}
\left|\omega' - \left(\frac{r}{r'} \right) \omega\right|^{-(n+2s)}
&= \sum\limits_{j=0}^{\infty} \alpha_{n,s,j}\left( \left( \frac{r}{r'} \right)^{2}-2 (\omega\cdot \omega') \left( \frac{r}{r'} \right) \right)^{j}\\
&= \sum\limits_{j=0}^{\infty} \alpha_{n,s,j} \left( \frac{r}{r'} \right)^j \left( \left( \frac{r}{r'} \right) - 2 (\omega'\cdot \omega)  \right)^{j},
\end{align*}
where $|\alpha_{n,s,j}|\leq C_{n,s}(1+j)^{s+n/2-1}$.
This series is converging absolutely as $0<\frac{r}{r'} \ll \frac{1}{3}$ and $\left| \left( \frac{r}{r'} \right) - 2 (\omega'\cdot \omega)  \right| \leq 3$. Moreover, as $ \left( \left( \frac{r}{r'} \right) - 2 (\omega'\cdot \omega)  \right)^{j}$ is a polynomial of degree $j$ in $\omega$, it has an expansion in terms of the spherical harmonics of degree less than $j$:
\begin{align*}
\left( \left( \frac{r}{r'} \right) - 2 (\omega'\cdot \omega)  \right)^{j}
= \sum\limits_{m=0}^{j} \sum\limits_{l=0}^{l_m} \mu_{l,r/r',\omega'} H_{m,l}(\omega).
\end{align*}
In particular, the orthogonality of the spherical harmonics hence implies
\begin{align*}
\int\limits_{S^{n-1}} \frac{H_k(\omega)}{ |r' \omega' - r \omega|^{n+ 2s} } \,d\omega
&= (r')^{-n-2s} \sum\limits_{j=0}^{\infty} \alpha_{n,s,j} \left( \frac{r}{r'} \right)^j  \int\limits_{S^{n-1}}  H_k(\omega) \left( \left( \frac{r}{r'} \right)- 2 (\omega'\cdot \omega) \right)^{j} \,d\omega\\
&= (r')^{-n-2s} \sum\limits_{j=0}^{k-1} \sum\limits_{m=0}^{j} \sum\limits_{l=0}^{l_m}  \alpha_{n,s,j} \mu_{l,r/r',\omega'} \left( \frac{r}{r'} \right)^j  \int\limits_{S^{n-1}}  H_k(\omega)  H_{m,l}(\omega)  \,d\omega\\
& \quad 
+ (r')^{-n-2s} \sum\limits_{j=k}^{\infty} \alpha_{n,s,j} \left( \frac{r}{r'} \right)^j  \int\limits_{S^{n-1}}  H_k(\omega) \left( \left( \frac{r}{r'} \right) - 2 (\omega'\cdot \omega)  \right)^{j} \,d\omega\\
& =  (r')^{-n-2s} \sum\limits_{j=k}^{\infty} \alpha_{n,s,j} \left( \frac{r}{r'} \right)^j  \int\limits_{S^{n-1}}  H_k(\omega) \left( \left( \frac{r}{r'} \right) - 2  (\omega' \cdot \omega) \right)^{j} \,d\omega.
\end{align*}
As a consequence, we estimate
\begin{align*}
\left| \,\int\limits_{S^{n-1}} \frac{H_k(\omega)}{ |r' \omega' - r \omega|^{n+ 2s} } \,d\omega \right|
\leq (r')^{-n-2s}\left| \sum\limits_{j=k}^{\infty} \alpha_{n,s,j} \left( \frac{3 r}{r'} \right)^j  \right| \|H_k\|_{L^1(S^{n-1})}
\leq C_{n,s}  2^{-k},
\end{align*}
where we used that $(r')^{-n-2s} \leq 1$, $\left| \frac{r}{r'} - 2 \omega'\cdot \omega \right|\leq 3$, $\|H_k\|_{L^1(S^{n-1})} \leq C_{n}\|H_k\|_{L^2(S^{n-1})}=C_{n}$ and 
\[
\left| \sum\limits_{j=k}^{\infty} \alpha_{n,s,j} \left( \frac{3 r}{r'} \right)^j \right| \leq 
2^{-k} \left| \sum\limits_{j=k}^{\infty} \alpha_{n,s,j} \left( \frac{6 r}{ r'} \right)^j \right| 
\leq 2^{-k} \sum\limits_{j=k}^{\infty} C_{n,s}(1+j)^{s+n/2-1} \left( \frac{6 r}{ r'} \right)^j \leq C_{n,s} 2^{-k}
\]
if $r'\gg r$ is so large that $\frac{6 r}{ r'} \leq 1/2$ (which implies the absolute convergence of the power series).
Hence,
\begin{align*}
\|h_k\|_{L^2(W)}
\leq \left|\int\limits_{0}^1 r^{n-1} j_k(r)^{2} \,dr \right|^{1/2} C_{n,s} 2^{-k} \|1\|_{L^2(W)}
\leq C_{n,R,s} 2^{-k}
\end{align*}
where the $r$-integral is $\leq 1$ using the normalization $\|v_k\|_{L^2(B_1)} = 1$. By duality, we obtain
\begin{align*}
\|h_k\|_{H^{-s}(W)} 
&= \sup\limits_{\|\varphi\|_{H^s_{\overline{W}}}=1} (h,\varphi)_{W}
\leq \sup\limits_{\|\varphi\|_{H^s_{\overline{W}}}=1} \|h\|_{L^2(W)}\|\varphi\|_{L^2(W)}\\
&\leq \sup\limits_{\|\varphi\|_{H^s_{\overline{W}}}=1} \|h\|_{L^2(W)}\|\varphi\|_{H^s_{\overline{W}}}
\leq \|h\|_{L^2(W)} \leq C_{n,R,s} 2^{-k}.
\end{align*}
This concludes the argument.
\end{proof}

\appendix

\section{Carleman Estimate}
\label{sec:append}
In this section, we provide a self-contained argument for the Carleman estimate, which was used in Section \ref{sec:ucp}.
Here we partially avoid the logarithmic losses from \cite{Rue15}, which in our application to doubling estimates is needed both on an $L^2$ and the gradient level:

\begin{prop}
\label{prop:Carl}
Let $s\in (0,1)$ and $n\geq 1$. Assume that $V\in L^{\infty}(B_5')$.
Let $w \in H^{1}(B_5^+, x_{n+1}^{1-2s})$ with $\supp(w) \subset \overline{B_4^+ \setminus B_{r_1}^+}$ for $r_1 \in (0,1)$ be a solution of
\begin{align}
\label{eq:inhom_app}
\begin{split}
\nabla \cdot x_{n+1}^{1-2s} \nabla w & = f \mbox{ in } B_5^+,\\
\lim\limits_{x_{n+1}\rightarrow 0} x_{n+1}^{1-2s} \p_{n+1} w & = Vw \mbox{ on } B_5'.
\end{split}
\end{align}
Then there exists a constant $\tau_0>1$ such that for any parameter $\tau \geq \tau_0$ and for the weight function $\phi(x):= \psi(|x|)$ with
\begin{align*}
\psi(r) = -\ln(r) + \frac{1}{10}\left( \ln(r) \arctan(\ln(r))- \frac{1}{2} \ln(1+ \ln^2(r))  \right),
\end{align*}
we have for any $r_2\in (2r_1,3)$
\begin{align}
\label{eq:Carl_app_a}
\begin{split}
& \tau^{\frac{1}{2}}\ln(r_2/r_1)^{-1}\|e^{\tau \phi} x_{n+1}^{\frac{1-2s}{2}} |x|^{-1} w \|_{L^2(B_{r_2}^+)}
+ \tau^{-\frac{1}{2}}\ln(r_2/r_1)^{-1} \|e^{\tau \phi} x_{n+1}^{\frac{1-2s}{2}} \nabla w\|_{L^2(B_{r_2}^+)}\\
& + \tau^s \|e^{\tau \phi} (1+\ln^2(|x|))^{-1/2} |x|^{-s} w\|_{L^2(B_5')}\\
&+ \tau \|e^{\tau \phi} (1+\ln^2(|x|))^{-1/2} x_{n+1}^{\frac{1-2s}{2}}|x|^{-1} w \|_{L^2(B_5^+)}
+  \|e^{\tau \phi}  (1+\ln^2(|x|))^{-1/2} x_{n+1}^{\frac{1-2s}{2}} \nabla w \|_{L^2(B_5^+)}\\
&\leq C \tau^{-\frac{1}{2}}\| e^{\tau \phi} |x| x_{n+1}^{\frac{2s-1}{2}} f\|_{L^2(B_5^+)} + \tau^{\frac{1-2s}{2}} \|e^{\tau \phi} |x|^{s} V w\|_{L^2(B_5')}.
\end{split}
\end{align}
\end{prop}

Here \eqref{eq:inhom_app} is interpreted in a weak sense similarly as in \eqref{eq:weak_1}, \eqref{eq:weak_2}. In particular, the boundary data are interpreted in this formal sense.

\begin{proof}
We first introduce conformal coordinates, $x=e^{t}\theta$, $t\in \R$, $\theta \in S^n_+$, and pass from the function $w(x)$ to the function $u(t,\theta):= e^{\frac{n-2s}{2}t} w(e^{t}\theta)$ (by conjugation with the corresponding weights). Since $w$ is assumed to be a solution of \eqref{eq:inhom_app}, the function $u$ solves the equation
\begin{align}
\label{eq:op_conf}
\begin{split}
\left( \theta_n^{1-2s}\p_t^2 - \theta_n^{1-2s}\frac{(n-2s)^2}{4} + \nabla_{S^n}\cdot \theta_n^{1-2s} \nabla_{S^n} \right) u & = \overline{f} \mbox{ in } \R \times S^{n}_+,\\
\lim\limits_{\theta_n \rightarrow 0} \theta_n^{1-2s} \nu \cdot \nabla_{S^n} u & =: \overline{h} \mbox{ on } \R \times S^{n-1},
\end{split}
\end{align}
where $\overline{h}(t,\theta)= e^{2st} V(e^t \theta) u(t,\theta)$, $\overline{f}(t,\theta) = e^{\frac{n+2+2s}{2}t}f(e^{t}\theta)$ and $\theta_n:= \frac{x_{n+1}}{|x|}$. Again, the equation \eqref{eq:op_conf} is here interpreted in a weak sense, where the boundary data are assumed to be formal, similarly as in \eqref{eq:weak}.
In these conformal coordinates and with $\varphi(t)=\psi(e^t)$, the Carleman estimate \eqref{eq:Carl_app_a} turns into
\begin{align}
\label{eq:Carl_ap}
\begin{split}
& \tau^{\frac{1}{2}} |t_2-t_1|^{-1}  \|e^{\tau \varphi} \theta_n^{\frac{1-2s}{2}} u \|_{L^2((t_1,t_2) \times S^n_+)} 
+\tau^{-\frac{1}{2}} |t_2 -t_1|^{-1}\|e^{\tau \varphi} \theta_n^{\frac{1-2s}{2}} \p_t u\|_{L^2((t_1,t_2) \times S^n_+)}\\
&+ \tau^{-\frac{1}{2}} |t_2 -t_1|^{-1}\|e^{\tau \varphi}\theta_n^{\frac{1-2s}{2}}
 \nabla_{S^n} u\|_{L^2((t_1,t_2) \times S^n_+)}\\
& + \tau^{s} \|e^{\tau \varphi}|\varphi''|^{\frac{1}{2}}  u\|_{L^2(\R \times S^{n-1})}
+ \tau \| e^{\tau \varphi}|\varphi''|^{\frac{1}{2}} \theta_n^{\frac{1-2s}{2}} u \|_{L^2(\R \times S^{n}_+)}\\
& +  \|e^{\tau \varphi}|\varphi''|^{\frac{1}{2}} \theta_n^{\frac{1-2s}{2}} \nabla_{S^n} u \|_{L^2(\R \times S^n_+)} 
+  \|e^{\tau \varphi}|\varphi''|^{\frac{1}{2}} \theta_n^{\frac{1-2s}{2}} \p_t u \|_{L^2(\R \times S^n_+)}\\
&\leq C \tau^{-\frac{1}{2}} \|e^{\tau \varphi}\theta_n^{\frac{2s-1}{2}} \overline{f}\|_{L^2(\R \times S^n_+)} + \tau^{\frac{1-2s}{2}} \|e^{\tau \varphi} \overline{h}\|_{L^2(\R \times S^{n-1})}.
\end{split}
\end{align}
Here $t_1= \ln(r_1)$, $t_2= \ln(r_2)$.
In order to prove \eqref{eq:Carl_ap}, we pursue a splitting strategy, separating the problem into an elliptic and a subelliptic estimate: More precisely, we set $u=u_1+ u_2$, where $u_1$ is a solution to
\begin{align}
\label{eq:u_1}
\begin{split}
\left(\theta_{n}^{1-2s} \p_t^2 + \nabla_{S^n} \cdot \theta_n^{1-2s} \nabla_{S^n} -\theta_n^{1-2s}\frac{(n-2s)^2}{4}- K^2 \tau^2 \theta_n^{1-2s}\right)u_1 &= \overline{f} \mbox{ in } \R \times S^{n}_+,\\
\lim\limits_{\theta_{n}\rightarrow 0} \theta_{n}^{1-2s} \p_{\nu} u_1 & =  \overline{h} \mbox{ on } \R \times S^{n-1}.
\end{split}
\end{align}
Here $K\geq 1$ is a large constant, whose precise value will be chosen later, and $\p_{\nu}:= \nu\cdot \nabla_{S^n}$. We note that this equation is elliptic (as an equation in the space variables and in the parameter $\tau$).
As a consequence, the function $u_2 = u-u_1$ solves
\begin{align}
\label{eq:u_2}
\begin{split}
\left(\theta_{n}^{1-2s} \p_t^2 + \nabla_{S^n} \cdot \theta_n^{1-2s} \nabla_{S^n}-\theta_n^{1-2s}\frac{(n-2s)^2}{4}\right)u_2 &= -K^2 \tau^2 \theta_n^{1-2s} u_1 \mbox{ in } \R \times S^{n}_+,\\
\lim\limits_{\theta_{n}\rightarrow 0} \theta_{n}^{1-2s} \p_{\nu} u_2 & = 0 \mbox{ on } \R \times S^{n-1}.
\end{split}
\end{align}
In the sequel, we derive separate estimates for $u_1$ and $u_2$: For $u_1$ we will use purely ``elliptic estimates", while for $u_2$ we use a subelliptic Carleman estimate.\\ 

\emph{Step 1: Estimate for $u_1$.}
We first comment on the well-posedness of the elliptic problem \eqref{eq:u_1}: The solvability of \eqref{eq:u_1} follows from the Lax-Milgram theorem. Indeed, the weak formulation of \eqref{eq:u_1} reads
\begin{align}
\label{eq:weak}
\begin{split}
&(\theta_n^{1-2s} \p_t u_1, \p_t \xi) + (\theta_n^{1-2s} \nabla_{S^n} u_1, \nabla_{S^n} \xi) + \frac{(n-2s)^2}{4}(\theta_n^{1-2s} u_1, \xi)
+ K^2 \tau^2 (\theta_n^{1-2s} u_1, \xi)\\
& = -(\bar{f},\xi) +(\bar{h},\xi)_0 \mbox{ for all } \xi \in H^{1}(\R \times S^n_+, \theta_n^{1-2s}),
\end{split}
\end{align}
where the index $0$ denotes the restriction of the $L^2$ inner product onto the boundary $\R \times S^{n-1}$ and where all the other inner products are on $\R \times S^{n}_+$.
In order to apply the Lax-Milgram theorem, we note that for $\overline{f}\in L^2(\R \times S^{n}_+, \theta_n^{2s-1})$ and $\overline{h} \in L^2(\R \times S^{n-1})$ the mapping
\begin{align*}
\xi \mapsto -(\bar{f},\xi) +(\bar{h},\xi)_0 
\end{align*}
is a bounded functional on $H^{1}(\R \times S^{n}_+, \theta_n^{1-2s})$. This follows from the estimates
\begin{align*}
&|(\overline{f},\xi)| 
\leq \|\theta_n^{\frac{2s-1}{2}} \overline{f}\|_{L^2}\|\theta_n^{\frac{1-2s}{2}} \xi\|_{L^2},\\
& |(\overline{h},\xi)_0| \leq \|\overline{h}\|_{0}\|\xi\|_{0}
\leq C \|\overline{h}\|_{0}(\|\theta_n^{\frac{1-2s}{2}} \xi\|_{L^2}
+ \|\theta_n^{\frac{1-2s}{2}} \nabla_{S^n} \xi\|_{L^2}),
\end{align*}
where we have used a (weighted) trace inequality to control the boundary term (c.f. for example Lemma 2.2. in \cite{FF14}). Thus, the Lax-Milgram theorem is applicable in the space $H^{1}(\R \times S^{n}_+, \theta_n^{1-2s})$ and yields the existence of a unique weak solution $u_1\in H^1(\R \times S^n_+, \theta_n^{1-2s})$ solving \eqref{eq:weak}. 
\\
We next observe that the solution $u_1$ decays fast as $|t|\rightarrow \infty$. This can be seen by considering test functions $e^{\tau \tilde{\varphi}}$, where $\tilde{\varphi}$ is only $t$ dependent, grows linearly for at least one of the limits $t \rightarrow  \pm \infty$ and $| \tilde{\varphi}'|, | \tilde{\varphi}''|<C$. 
Indeed, let 
\begin{align*}
\tilde{\varphi}_{M,\delta}= \min\{M, \max\{-M,\tilde{\varphi}\}\}\ast \gamma_{\delta},
\end{align*}
be a mollified truncation of $\tilde{\varphi}$,
where $\gamma_{\delta}$ is a standard mollifier and $M\gg 1$ denotes a large constant (which will be sent to $\infty$ later). 
We test the equation \eqref{eq:u_1} (or equivalently \eqref{eq:weak}) with the function $e^{2\tau \tilde{\varphi}_{M,\delta}} u_1$ and derive the corresponding energy estimates. This yields
\begin{align*}
&(\theta_{n}^{1-2s} \p_t u_1, \p_t (e^{2\tau \tilde{\varphi}_{M,\delta}} u_1))
+ (\theta_{n}^{1-2s} \nabla_{S^n} u_1 , \nabla_{S^n}(e^{2\tau \tilde{\varphi}_{M,\delta}} u_1)) 
-(\theta_{n}^{1-2s}\p_{\nu} u_1, e^{2\tau \tilde{\varphi}_{M,\delta}}u_1)_0 \\
&+ \frac{(n-2s)^2}{4}(\theta_n^{1-2s}u_1, e^{2\tau \tilde{\varphi}_{M,\delta}}u_1) + K^2 \tau^2( \theta_{n}^{1-2s} u_1, e^{2\tau \tilde{\varphi}_{M,\delta}}u_1)
 = - (\overline{f}, e^{2\tau \tilde{\varphi}_{M,\delta}} u_1).
\end{align*}
Carrying out the associated integration by parts (using that $\tilde{\varphi}$ only depends on $t$), we infer
\begin{align}
\label{eq:ell}
\begin{split}
&\|e^{\tau \tilde{\varphi}_{M,\delta}}\theta_{n}^{\frac{1-2s}{2}} \p_t u_1\|_{L^2}^2 +
\|e^{\tau \tilde{\varphi}_{M,\delta}}\theta_{n}^{\frac{1-2s}{2}} \nabla_{S^n} u_1\|_{L^2}^2  +
K^2\tau^2 \|e^{\tau \tilde{\varphi}_{M,\delta}}\theta_{n}^{\frac{1-2s}{2}} u_1\|_{L^2}^2\\
& \quad  + \frac{(n-2s)^2}{4}\|e^{\tau \tilde{\varphi}_{M,\delta}}\theta_{n}^{\frac{1-2s}{2}} u_1\|_{L^2}^2
+ 2\tau(e^{\tau \tilde{\varphi}_{M,\delta}}\theta_{n}^{1-2s} \p_t u_1, (\p_t \tilde{\varphi}_{M,\delta}) e^{\tau \tilde{\varphi}_{M,\delta}} u_1)\\
&\quad -(e^{\tau \tilde{\varphi}_{M,\delta}}\theta_{n}^{1-2s}\p_{\nu} u_1, e^{\tau \tilde{\varphi}_{M,\delta}}u_1)_0
= -(e^{\tau \tilde{\varphi}_{M,\delta}} \overline{f} , e^{\tau \tilde{\varphi}_{M,\delta}}u_1).
\end{split}
\end{align}
The first four terms are positive. The other ones are in general unsigned, however by our choice of the weight function (and by choosing $K\geq 1$ large enough) they can be controlled by the positive contributions in the first line. The term on the right hand side can be bounded by the inhomogeneity $\overline{f}$ and by terms which can be absorbed into the left hand side of \eqref{eq:ell}:
\begin{align*}
|(e^{\tau \tilde{\varphi}_{M,\delta}} \overline{f} , e^{\tau \tilde{\varphi}_{M,\delta}}u_1)|
\leq \tau^{-2}\|e^{\tau \tilde{\varphi}_{M,\delta}} \theta_{n}^{\frac{2s-1}{2}}\overline{f}\|_{L^2}^2 + \tau^2 \|e^{\tau \tilde{\varphi}_{M,\delta}} \theta_{n}^{\frac{1-2s}{2}} u_1\|_{L^2}^2.
\end{align*}
The boundary term in \eqref{eq:ell} is controlled by 
\begin{align}
\label{eq:boundary}
\begin{split}
|(e^{\tau \tilde{\varphi}_{M,\delta}}\theta_{n}^{1-2s}\p_{\nu} u_1, e^{\tau \tilde{\varphi}_{M,\delta}}u_1)_0| \leq 
C_{\epsilon} \tau^{-2s}\|e^{\tau \tilde{\varphi}_{M,\delta}}\lim\limits_{\theta_n \rightarrow 0}\theta_{n}^{1-2s}\p_{\nu} u_1\|_{0}^2 + \epsilon \tau^{2s}\| e^{\tau \tilde{\varphi}_{M,\delta}}u_1\|_0^2
\end{split}
\end{align}
for a suitably small constant $\epsilon>0$.
Absorbing the bulk error terms (for a sufficiently large choice of $K>1$), we hence infer
\begin{align}
\label{eq:ell_1}
\begin{split}
&\tau\|e^{\tau \tilde{\varphi}_{M,\delta}}\theta_{n}^{\frac{1-2s}{2}} \p_t u_1\|_{L^2} +
\tau \|e^{\tau \tilde{\varphi}_{M,\delta}}\theta_{n}^{\frac{1-2s}{2}} \nabla_{S^n} u_1\|_{L^2}  +
\frac{K}{2}\tau^2 \|e^{\tau \tilde{\varphi}_{M,\delta}}\theta_{n}^{\frac{1-2s}{2}} u_1\|_{L^2}
\\
&\leq  \|e^{\tau \tilde{\varphi}_{M,\delta}} \theta_{n}^{\frac{1-2s}{2}} \overline{f}\|_{L^2}
+ C_{\epsilon}\tau^{1-s}\|e^{\tau \tilde{\varphi}_{M,\delta}}\lim\limits_{\theta_{n}\rightarrow 0}\theta_{n}^{1-2s}\p_{\nu} u_1\|_{0} + \epsilon\tau^{1+s}\| e^{\tau \tilde{\varphi}_{M,\delta}}u_1\|_0.
\end{split}
\end{align}
By the boundary bulk interpolation estimate from \cite{Rue15}, i.e., by the estimate
\begin{align}
\label{eq:boundary_bulk_trace}
\|u\|_{L^2(S^{n-1})} \leq C(\tau^{1-s}\|\theta_n^{\frac{1-2s}{2}}u\|_{L^2(S^n_+)} + \tau^{-s}\|\theta_n^{\frac{1-2s}{2}}\nabla_{S^n} u\|_{L^2(S^n_+)}),
\end{align}
we can further add boundary terms onto the left hand side of \eqref{eq:ell_1} and infer 
\begin{align*}
\begin{split}
&\tau\|e^{\tau \tilde{\varphi}_{M,\delta}}\theta_{n}^{\frac{1-2s}{2}} \p_t u_1\|_{L^2} +
\tau \|e^{\tau \tilde{\varphi}_{M,\delta}}\theta_{n}^{\frac{1-2s}{2}} \nabla_{S^n} u_1\|_{L^2}  +
\frac{K}{2}\tau^2 \|e^{\tau \tilde{\varphi}_{M,\delta}}\theta_{n}^{\frac{1-2s}{2}} u_1\|_{L^2}
+ C \tau^{1+s}\|e^{\tau \tilde{\varphi}_{M,\delta}}u_1\|_0
\\
&\leq  \|e^{\tau \tilde{\varphi}_{M,\delta}} \theta_{n}^{\frac{1-2s}{2}} \overline{f}\|_{L^2}
+ C_{\epsilon}\tau^{1-s}\|e^{\tau \tilde{\varphi}_{M,\delta}}\lim\limits_{\theta_n \rightarrow 0}\theta_{n}^{1-2s}\p_{\nu} u_1\|_{0} + \epsilon \tau^{1+s}\| e^{\tau \tilde{\varphi}_{M,\delta}}u_1\|_0.
\end{split}
\end{align*}
For $\epsilon>0$ sufficiently small, we can thus absorb the zeroth order boundary term, which leads to
\begin{align}
\label{eq:ell_2}
\begin{split}
&\tau\|e^{\tau \tilde{\varphi}_{M,\delta}}\theta_{n}^{\frac{1-2s}{2}} \p_t u_1\|_{L^2} +
\tau \|e^{\tau \tilde{\varphi}_{M,\delta}}\theta_{n}^{\frac{1-2s}{2}} \nabla_{S^n} u_1\|_{L^2}  +
\frac{K}{2}\tau^2 \|e^{\tau \tilde{\varphi}_{M,\delta}}\theta_{n}^{\frac{1-2s}{2}} u_1\|_{L^2}
+ C \tau^{1+s}\|e^{\tau \tilde{\varphi}_{M,\delta}}u_1\|_0
\\
&\leq  \|e^{\tau \tilde{\varphi}_{M,\delta}} \theta_{n}^{\frac{1-2s}{2}} \overline{f}\|_{L^2}
+ C_{\epsilon}\tau^{1-s}\|e^{\tau \tilde{\varphi}_{M,\delta}}\overline{h}\|_{0}.
\end{split}
\end{align}
We remark that here and in the sequel, the constant $C>1$ may change from line to line without further comment, but does not depend on $\tau>0$.
With \eqref{eq:ell_2} at hand and by using the compact supports of the data $\overline{f}$, $\overline{h}$ we pass to the limits $M\rightarrow \infty$ and $\delta \rightarrow 0$. On the one hand, this proves the claimed fast decay of $u_1$ at $|t|\rightarrow \infty$ and thus in particular allows us to formulate Carleman estimates for the function $u_1$ with linearly growing weight functions. On the other hand, by choosing $\tilde{\varphi}= \varphi$, we deduce
\begin{align}
\label{eq:ell_3}
\begin{split}
&\tau\|e^{\tau \varphi}\theta_{n}^{\frac{1-2s}{2}} \p_t u_1\|_{L^2} +
\tau \|e^{\tau \varphi}\theta_{n}^{\frac{1-2s}{2}} \nabla_{S^n} u_1\|_{L^2}  +
\frac{K}{2}\tau^2 \|e^{\tau \varphi}\theta_{n}^{\frac{1-2s}{2}} u_1\|_{L^2}
+ C \tau^{1+s}\|e^{\tau \varphi}u_1\|_0
\\
&\leq  \|e^{\tau \varphi} \theta_{n}^{\frac{1-2s}{2}} \overline{f}\|_{L^2}
+ C_{\epsilon}\tau^{1-s}\|e^{\tau \varphi}\overline{h}\|_{0},
\end{split}
\end{align}
which proves the exponentially weighted estimate for $u_1$.\\

\emph{Step 2: Commutator estimate for $u_2$.}\\
We derive the estimate for $u_2$ by the usual conjugation argument and exploit the fact that $u_2$ has vanishing Neumann data. More precisely, we start by conjugating the equation \eqref{eq:u_2} by $\theta_n^{\frac{2s-1}{2}}$. This separates the spherical and the radial variables. Defining $\tilde{u}_2:= \theta_n^{\frac{1-2s}{2}}u_2$, we obtain the equation
\begin{align}
\label{eq:u22}
\begin{split}
(\p_t^2 + \theta_n^{\frac{2s-1}{2}} \nabla_{S^n} \cdot \theta_n^{1-2s} \nabla_{S^n} \theta_n^{\frac{2s-1}{2}} - \frac{(n-2s)^2}{4}) \tilde{u}_2 & = - K^2 \tau^2 \theta_n^{\frac{1-2s}{2}} u_1 \mbox{ in } \R\times S^{n}_+,\\
\lim\limits_{\theta_n \rightarrow 0} \theta_n^{1-2s} \p_{\nu} \theta_n^{\frac{2s-1}{2}} \tilde{u}_2 & = 0 \mbox{ on } \R \times S^{n-1}.
\end{split}
\end{align}
Next we conjugate the problem with the Carleman weight. To this end, we set $v_2 = e^{\tau \varphi} \tilde{u}_2$. We remark that by the fast decay of the auxiliary function $u_1$ (c.f. the truncation argument in Step 1) and the compact support of the original function $u$, also the auxiliary function $u_2$ has fast decay as $|t|\rightarrow \infty$. In particular, this yields that $v_2 \in H^1(\R \times S^n_+, \theta_n^{1-2s})$ and allows us to formulate the corresponding Carleman estimates for $u_2$. The function $v_2$ solves the equation
\begin{align}
\label{eq:u222}
\begin{split}
(\p_t^2 + \tau^2|\varphi'|^2 - 2\tau \varphi' \p_t - \tau  \varphi'' + \tilde{\Delta}_{S^{n}}-\frac{(n-2s)^2}{4}) v_2 & = - K^2 \tau^2 \theta_n^{\frac{1-2s}{2}} e^{\tau \varphi}u_1 \mbox{ in } \R\times S^{n}_+,\\
\lim\limits_{\theta_n \rightarrow 0} \theta_n^{1-2s} \p_{\nu} \theta_n^{\frac{2s-1}{2}} v_2 & = 0 \mbox{ on } \R \times S^{n-1},
\end{split}
\end{align}
in a weak sense,
where for ease of notation we have abbreviated $\tilde{\Delta}_{S^n}:= \theta_n^{\frac{2s-1}{2}} \nabla_{S^n} \cdot \theta_n^{1-2s} \nabla_{S^n} \theta_n^{\frac{2s-1}{2}}$.
Up to boundary terms, this yields the splitting into the symmetric and antisymmetric parts of the operator:
\begin{align*}
S &= \p_t^2 + \tau^2 |\varphi'|^2 + \tilde{\D}_{S^n} - \frac{(n-2s)^2}{4},\\
A &= -2 \tau \varphi' \p_t - \tau \varphi''.
\end{align*}
Hence, the commutator becomes
\begin{align*}
[S,A] = - 4 \tau \varphi'' \p_t^2 - 2\tau \varphi''' \p_t - 2\tau \varphi''' \p_t - \tau \varphi'''' + 4 \tau^3 (\varphi')^2 \varphi''.
\end{align*}
Thus, after integrating by parts, we obtain that for $L_{\varphi}=S+A$
\begin{align}
\label{eq:Carl_conj}
\begin{split}
\|L_{\varphi}v_2\|_{L^2}^2
& = \|A v_2\|_{L^2}^2 + \|Sv_2\|_{L^2}^2 + ([S,A]v_2, v_2)\\ 
& \quad - 2\tau (\varphi' \lim\limits_{\theta_n \rightarrow 0} \theta_n^{\frac{2s-1}{2}}\p_t v_2, \lim\limits_{\theta_n\rightarrow 0} \theta_n^{1-2s}\p_{\nu} \theta_n^{\frac{2s-1}{2}}v_2)_0\\
& \quad + 2\tau (\lim\limits_{\theta_n \rightarrow 0} \theta_n^{1-2s} \p_{\nu} \theta_n^{\frac{2s-1}{2}}\p_t v_2, \lim\limits_{\theta_n \rightarrow 0} \theta_n^{\frac{1-2s}{2}} v_2)_0.
\end{split}
\end{align}

We note that both boundary terms are well-defined (for the first one, this follows from the assumption that the Neumann derivative is in $L^2$, for the second one, it follows from the fact that the equation \eqref{eq:u_2} can be differentiated with respect to the tangential direction with a controlled right hand side) and that by the zero Neumann boundary conditions for $v_2$ both contributions vanish. Thus, inserting the expression for the commutator into \eqref{eq:Carl_conj} and integrating by parts in the tangential direction, we further infer
\begin{align*}
\begin{split}
\|L_{\varphi}v_2\|_{L^2}^2
& = \|A v_2\|_{L^2}^2 + \|Sv_2\|_{L^2}^2 +
4 \tau (\varphi'' \p_t v_2, \p_t v_2) + 4 \tau^3 ((\varphi')^2 (\varphi'') v_2,v_2)\\
& \quad -2 \tau (\varphi''' \p_t v_2, v_2) - \tau (\varphi'''' v_2, v_2).
\end{split}
\end{align*}
We remark that by density and approximation arguments in weighted Sobolev spaces, c.f. for instance \cite{CK14}, 
we may assume that $u_1 \in C_c^{\infty}(\mathbb{R}^{n+1}_+)$, so that we can invoke regularity results similar as for the Neumann problem from the Appendix of \cite{KRS16}, in order to make sense of the second derivative contributions appearing in the integration by parts estimates. 
Using the explicit expression for $\varphi$ (and choosing $\tau \geq \tau_0$ for some sufficiently large constant $\tau_0>1$), the non-positive terms can be absorbed into the positive commutator contributions, yielding the bound
\begin{align}
\label{eq:Carl_conj0001}
\begin{split}
\|L_{\varphi}v_2\|_{L^2}^2
& \geq \|A v_2\|_{L^2}^2 + \|Sv_2\|_{L^2}^2 +
 3\tau \||\varphi''|^{1/2} \p_t v_2\|_{L^2}^2 + 3 \tau^3 \|(\varphi') |\varphi''|^{1/2} v_2\|_{L^2}^2.
\end{split}
\end{align}
By using the symmetric part, it is further possible to also upgrade this to a full gradient estimate (i.e. to include the spherical gradient). Although a similar argument will be used in Step 3b below, we discuss the details: 
Spelling out the symmetric part $S$, testing with $\varphi'' v_2$ and using the explicit form of the Carleman weight as well as the bound \eqref{eq:Carl_conj0001} yields (for a sufficiently small constant $c>0$) 
\begin{align*}
c \tau \||\varphi''|^{1/2} \theta_n^{\frac{1-2s}{2}}\nabla_{S^n} \theta_n^{\frac{2s-1}{2}}  v_2 \|_{L^2}^2
&\leq c \tau |(|\varphi''|^{1/2} S v_2, v_2)| + c \tau \||\varphi''|^{1/2} \p_t v_2\|_{L^2}^2 + c \tau^3 \||\varphi''|^{1/2} \varphi' v_2\|_{L^2}^2\\
& \quad + c \tau \frac{(n-2s)^2}{4} \||\varphi''|^{1/2}v_2\|_{L^2}^2\\
&\leq c \|S v_2\|_{L^2}^2 +  c\tau^2 \||\varphi''|^{1/2}  v_2\|_{L^2}^2 + 2c \tau \||\varphi''|^{1/2} \p_t v_2\|_{L^2}^2 \\
& \quad + 2c \tau^3 \||\varphi''|^{1/2} \varphi' v_2\|_{L^2}^2 
 + 2c \tau \frac{(n-2s)^2}{4} \||\varphi''|^{1/2}v_2\|_{L^2}^2\\
&\leq  \|L_{\varphi} v_2\|_{L^2}^2.
\end{align*}

Moreover, combined with the boundary-bulk trace inequality \eqref{eq:boundary_bulk_trace}, this implies
\begin{align}
\label{eq:Carl_conj1}
\begin{split}
C \|L_{\varphi}v_2\|_{L^2}
& \geq \|A v_2\|_{L^2} + \|Sv_2\|_{L^2} \\
& \qquad +
\tau^{1/2} \||\varphi''|^{1/2} \tilde{\nabla} v_2\|_{L^2} +  \tau^{3/2} \|(\varphi') |\varphi''|^{1/2} v_2\|_{L^2} + \tau^{\frac{1}{2}+s}\||\varphi''|^{1/2}v_2\|_0,
\end{split}
\end{align}
where we abbreviated $\tilde{\nabla}:=(\p_t, \theta_n^{\frac{1-2s}{2}} \nabla_{S^n} \theta_n^{\frac{2s-1}{2}})$.
\\

\emph{Step 3: Derivation of the strengthened bulk estimates without logarithmic losses.}

Due to the convexification of the Carleman weight $\varphi$ (which gives rise to logarithmic errors in the form of $\varphi''$), the commutator term does not directly control the first two contributions in \eqref{eq:Carl_app_a} (when applied to $u_2$). To infer this additional control, we directly exploit the antisymmetric and symmetric terms of the operator.\\

\emph{Step 3a: Dealing with the $L^2$ terms.}
We begin by deriving the additional $L^2$ contribution in \eqref{eq:Carl_app_a}. This follows as in \cite[Remark 4]{Rue15} by studying the contribution $\| Av_2 \|_{L^2}$ in \eqref{eq:Carl_conj1}. Using the triangle inequality and setting $\overline{w}:= \theta_n^{\frac{2s-1}{2}} v_2 = e^{\tau \varphi} u_2$ and $v_1:= e^{\tau \varphi}u_1$, where $u_1$ is the function from Step 1, we have $\overline{w} + v_1 = e^{\tau \varphi} u$ and we obtain
\begin{align}
\label{eq:anti_app}
\begin{split}
c_0\|\theta_n^{\frac{1-2s}{2}} A \overline{w}\|_{L^2(\R \times S^n_+)}
& \geq c_0\|\theta_n^{\frac{1-2s}{2}} A (\overline{w}+v_1)\|_{L^2(\R \times S^n_+)} - c_0\|\theta_n^{\frac{1-2s}{2}} A v_1 \|_{L^2(\R \times S^n_+)}\\
&\geq 2c_0\tau\|\varphi' \theta_{n}^{\frac{1-2s}{2}} \p_t (\overline{w}+v_1)\|_{L^2(\R\times S^n_+)} - c_0\tau \|\varphi'' \theta_n^{\frac{1-2s}{2}} (\overline{w}+v_1)\|_{L^2(\R \times S^n_+)}\\
& \quad  - 2c_0\tau \|\theta_n^{\frac{1-2s}{2}} \p_t v_1 \|_{L^2(\R \times S^n_+)}- c_0\tau \|\varphi'' \theta_n^{\frac{1-2s}{2}} v_1\|_{L^2(\R \times S^n_+)}
\\
&\geq 2c_0 \tau \| \varphi' \theta_n^{\frac{1-2s}{2}} \p_t (\overline{w}+v_1) \|_{L^2((t_1,t_2+1)\times S^n_+)}  
- c_0\tau \|\varphi'' \theta_n^{\frac{1-2s}{2}} \overline{w} \|_{L^2(\R \times S^n_+)}\\
& \quad - 2c_0\tau \|\theta_n^{\frac{1-2s}{2}} \p_t v_1 \|_{L^2(\R \times S^n_+)}- 2 c_0\tau \|\varphi'' \theta_n^{\frac{1-2s}{2}}  v_1\|_{L^2(\R \times S^n_+)}.
\end{split}
\end{align}
Since (for $c_0>0$ sufficiently small) the second contribution can be controlled by the commutator terms in the Carleman inequality for $v_2$ (i.e. by the terms in the second line of \eqref{eq:Carl_conj1}) and since the third and fourth terms are controlled by the Carleman inequality from \eqref{eq:ell_3}, 
we only consider the first term on the right hand side of \eqref{eq:anti_app} in more detail. 
Using the form of $\varphi'$ and the support assumption on 
$\overline{w}+v_1$ in connection with Poincar\'e's inequality yields 
\begin{align}
\label{eq:anti1}
\begin{split}
\| \varphi' \theta_n^{\frac{1-2s}{2}} \p_t (\overline{w}+v_1)\|_{L^2((t_1, t_2+1)\times S^n_+)}
&\geq  \frac{1}{2} \| \theta_n^{\frac{1-2s}{2}} \p_t (\overline{w}+v_1)\|_{L^2((t_1, t_2+1)\times S^n_+)}\\
&\geq C^{-1} |t_1-t_2|^{-1} \| \theta_n^{\frac{1-2s}{2}} (\overline{w}+v_1) \|_{L^2((t_1, t_2+1) \times S^n_+)},
\end{split}
\end{align}
where we have used that $|t_1-t_2|\geq 1$.
This hence implies the desired control on the first contribution in \eqref{eq:Carl_ap}. 
In particular, returning to Cartesian coordinates yields the desired $L^2$ contribution, i.e., the control on the first term on the left hand side of \eqref{eq:Carl_app_a} (applied to $u_2$).\\

\emph{Step 3b: Dealing with the gradient term.}
Next we seek to deduce the claimed improved control on the gradient, i.e. we seek to derive the estimate for the second term on the left hand side of \eqref{eq:Carl_app_a} (applied to $u_2$). This is split into two parts: The antisymmetric part yields improved control on the radial component of the gradient, while the symmetric part yields improved control on the spherical part of the gradient. Indeed, using the expression for the antisymmetric part, we directly obtain from \eqref{eq:anti_app} that 
\begin{align}
\label{eq:anti12}
\begin{split}
c_1\|\theta_n^{\frac{1-2s}{2}} A \overline{w}\|_{L^2(\R \times S^n_+)}
&\geq 2c_1 \tau \| \varphi' \theta_n^{\frac{1-2s}{2}} \p_t (\overline{w}+v_1) \|_{L^2((t_1,t_2+1)\times S^n_+)}  
- c_1\tau \|\varphi'' \theta_n^{\frac{1-2s}{2}} \overline{w} \|_{L^2(\R \times S^n_+)}\\
& \quad - 2c_1\tau \|\theta_n^{\frac{1-2s}{2}} \p_t v_1 \|_{L^2(\R \times S^n_+)}- 2 c_1\tau \|\varphi'' \theta_n^{\frac{1-2s}{2}}  v_1\|_{L^2(\R \times S^n_+)}.
\end{split}
\end{align}
Since for $c_1 > 0$ small the second, third and fourth contributions can again be absorbed into the positive commutator terms in \eqref{eq:Carl_conj1} and the estimate \eqref{eq:ell_3} respectively, we obtain the desired bound for the radial part of the gradient of $\theta_n^{\frac{1-2s}{2}}(\overline{w}+v_1)$, i.e.\ we also control the second contribution in the first line of \eqref{eq:Carl_ap} (even with $\tau^{1/2}$ in front of this term instead of $\tau^{-1/2}$). By the triangle inequality and the gradient estimates from \eqref{eq:ell_3} this also entails estimates for $\theta_n^{\frac{1-2s}{2}}\p_t \overline{w}$.

In order to infer the remaining control on the spherical part of the gradient (i.e.\ on the contribution 
$\|\theta_n^{\frac{1-2s}{2}}\nabla_{S^n} \overline{w}\|_{L^2(\R \times S^n_+)}$),
we rely on the symmetric part of the operator. More precisely, for some smooth cut-off function $\chi$ which only depends on $t$, which is supported in $(-\infty, t_2+1)$ and which is equal to one on $(-\infty,t_2)$ (the purpose of the cut-off function here is to only produce $L^2$ terms, which are controlled by the $L^2$ term from \eqref{eq:anti_app}) we obtain
\begin{align}
\label{eq:symm_app}
\begin{split}
-(S \theta_n^{\frac{1-2s}{2}} \overline{w}, \theta_n^{\frac{1-2s}{2}} \overline{w} \chi)_{L^2(\R \times S^{n}_+)}
&= (\p_t \overline{w}, \theta_n^{1-2s}\p_t(\overline{w} \chi))_{L^2(\R \times S^n_+)}
+ \frac{(n-2s)^2}{4}(\overline{w},\theta_n^{1-2s}\overline{w} \chi)_{L^2(\R \times S^n_+)}\\
& \quad - \tau^2((\varphi')^2 \overline{w}, \theta_n^{1-2s}\overline{w} \chi)_{L^2(\R \times S^n_+)}
 + (\nabla_{S^{n}} \overline{w}, \theta_n^{1-2s} \nabla_{S^{n}}\overline{w} \chi)_{L^2(\R \times S^n_+)}\\
& \quad + (\lim\limits_{\theta_n \rightarrow 0}\theta_n^{1-2s} \nu \cdot \nabla_{S^{n-1}}  \overline{w}, \lim\limits_{\theta_n \rightarrow 0}  \overline{w}\chi)_{0}.
\end{split}
\end{align}
If multiplied by $c_2 |t_1-t_2|^{-2} >0$ (where $c_2>0$ is a sufficiently small constant), the non-positive terms are controlled by terms, which are already present on the left hand side of the Carleman inequality (i.e. by terms from \eqref{eq:Carl_conj1} and by terms from Step 3a). More precisely, (after multiplying the whole expression in \eqref{eq:symm_app} by $c_2 |t_1-t_2|^{-2}$) the only non-positive bulk terms are given by
\begin{align}
\label{eq:non_pos}
-\frac{c_2}{2}|t_1-t_2|^{-2} (\overline{w}, \theta_n^{1-2s} \overline{w} \chi'')_{L^2(\R \times S^n_+)}
-c_2|t_1 -t_2|^{-2}\tau^2((\varphi')^2 \overline{w}, \theta_n^{1-2s}\overline{w})_{L^2(\R \times S^{n}_+)},
\end{align}
which can be absorbed into the left hand side of \eqref{eq:anti1} if $c_2 >0$ is sufficiently small (we remark that the first term in \eqref{eq:non_pos} comes from the term $(\p_t \overline{w}, \theta_n^{1-2s} \p_t(\overline{w} \chi))_{L^2(\R \times S^n_+)}$ if the differentiation falls on $\chi$ and a subsequent integration by parts). The boundary term in \eqref{eq:symm_app} is well-defined and vanishes, which can be seen by an argument similar as the one used in Step 2. Hence, we obtain that
\begin{align*}
&c_2 |t_1-t_2|^{-2} \|\theta_n^{\frac{1-2s}{2}} \nabla_{S^{n}} \overline{w} \|_{L^2((t_1,t_2)\times S^n_+)}^2\\
&\leq c_2 |t_1-t_2|^{-2}|(S \theta_n^{\frac{1-2s}{2}} \overline{w}, \theta_n^{\frac{1-2s}{2}} \overline{w} \chi)_{L^2(\R \times S^n_+)}| + Cc_2 \| \theta_n^{\frac{1-2s}{2}} A \overline{w}\|_{L^2(\R \times S^n_+)}^2\\
& \quad + Cc_2 |t_2-t_1|^{-2} |([S,A] \theta_n^{\frac{1-2s}{2}} \overline{w}, \theta_n^{\frac{1-2s}{2}} \overline{w})_{L^2(\R \times S^n_+)}|\\
& \quad +  C\|e^{\tau \varphi} \theta_{n}^{\frac{1-2s}{2}} \overline{f}\|_{L^2(\R \times S^n_+)}
+ C_{\epsilon}\tau^{1-s}\|e^{\tau \varphi}\overline{h}\|_{0}
\\
&\leq c_2 |t_1-t_2|^{-2} \|S \theta_n^{\frac{1-2s}{2}} \overline{w}\|_{L^2(\R \times S^n_+)}^2 
+ c_2 |t_1-t_2|^{-2} \|\theta_n^{\frac{1-2s}{2}} \overline{w}\|_{L^2(\R \times S^n_+)}^2 \\
& \quad + Cc_2 \|\theta_n^{\frac{1-2s}{2}} A \overline{w}\|_{L^2(\R \times S^n_+)}^2 + C c_2 |t_2-t_1|^{-2} |([S,A] \theta_n^{\frac{1-2s}{2}} \overline{w}, \theta_n^{\frac{1-2s}{2}} \overline{w})_{L^2(\R \times S^n_+)}|\\
& \quad +  C\|e^{\tau \varphi} \theta_{n}^{\frac{1-2s}{2}} \overline{f}\|_{L^2(\R \times S^n_+)}
+ C_{\epsilon}\tau^{1-s}\|e^{\tau \varphi}\overline{h}\|_{0}
\\
& \leq C \|L_{\varphi} \theta_n^{\frac{1-2s}{2}} \overline{w}\|_{L^2(\R \times S^n_+)}+  C\|e^{\tau \varphi} \theta_{n}^{\frac{1-2s}{2}} \overline{f}\|_{L^2(\R \times S^n_+)}
+ C_{\epsilon}\tau^{1-s}\|e^{\tau \varphi}\overline{h}\|_{0}.
\end{align*}
In particular this contribution is controlled by the Carleman inequality (by assumption we have that $|t_1-t_2|\geq 1$), we therefore also obtain that the full gradient term
\begin{align*}
c_2 |t_1-t_2|^{-1} \|\theta_n^{\frac{1-2s}{2}}\nabla_{S^{n}} \overline{w}\|_{L^2((t_1,t_2)\times S^n_+)} + c_2 |t_1-t_2|^{-1} \|\theta^{\frac{1-2s}{2}}_n \p_t \overline{w}\|_{L^2((t_1,t_2)\times S^n_+)}
\end{align*}
is controlled by the Carleman inequality. This concludes the argument for adding the two terms on the left hand side of \eqref{eq:Carl_app_a} (for $u_2$).\\

\emph{Step 4: Combination of Steps 1-3.}
We combine the estimates for $u_1$ and $u_2$. By virtue of the triangle inequality and abbreviating $\nabla = (\p_t, \nabla_{S^n})$, this yields
\begin{align}
\label{eq:sub_ell}
\begin{split}
& \tau^{1/2}\|e^{\tau \varphi}(\varphi'')^{1/2}\theta_{n}^{\frac{1-2s}{2}} \nabla u\|_{L^2}   +
\tau^{3/2} \|e^{\tau \varphi}(\varphi'')^{1/2} \theta_{n}^{\frac{1-2s}{2}} u\|_{L^2}\\
&+ C \tau^{\frac{1}{2}+s}\|e^{\tau \varphi}|\varphi''|^{1/2}u\|_0
 + |t_2-t_1|^{-1}\|e^{\tau \varphi}\tilde{\chi} \theta_n^{\frac{1-2s}{2}} u\|_{L^2}
+ |t_2-t_1|^{-1}\|e^{\tau \varphi}\tilde{\chi}\theta_n^{\frac{1-2s}{2}} \nabla u\|_{L^2}
\\
&\leq
\tau^{1/2}\|e^{\tau \varphi}|\varphi''|^{\frac{1}{2}}\theta_{n}^{\frac{1-2s}{2}} \nabla u_1\|_{L^2}  +
\tau^{3/2} \|e^{\tau \varphi}|\varphi''|^{\frac{1}{2}}\theta_{n}^{\frac{1-2s}{2}} u_1\|_{L^2}
+  \tau^{\frac{1}{2}+s}\|e^{\tau \varphi} |\varphi''|^{1/2} u_1\|_0\\
& \quad + |t_2-t_1|^{-1}\|e^{\tau \varphi}\tilde{\chi}\theta_n^{\frac{1-2s}{2}} u_1\|_{L^2}
+ |t_2-t_1|^{-1}\|e^{\tau \varphi}\tilde{\chi}\theta_n^{\frac{1-2s}{2}} \nabla u_1\|_{L^2}\\
&\quad + 
\tau^{1/2} \|e^{\tau \varphi} |\varphi''|^{\frac{1}{2}}\theta_{n}^{\frac{1-2s}{2}} \nabla u_2\|_{L^2}  + \tau^{3/2} \|e^{\tau \varphi} |\varphi''|^{\frac{1}{2}} \theta_{n}^{\frac{1-2s}{2}} u_2\|_{L^2}+ \tau^{\frac{1}{2}+s}\|e^{\tau \varphi} |\varphi''|^{1/2} u_2\|_0\\
& \quad
+ |t_2-t_1|^{-1}\|e^{\tau \varphi} \tilde{\chi} \theta_n^{\frac{1-2s}{2}} u_2\|_{L^2}
+ |t_2-t_1|^{-1}\|e^{\tau \varphi} \tilde{\chi} \theta_n^{\frac{1-2s}{2}} \nabla u_2\|_{L^2}
\\
&\leq
C \|e^{\tau \varphi} \theta_{n}^{\frac{2s-1}{2}} \overline{f}\|_{L^2}
+ C \tau^{1-s}\|e^{\tau \varphi} \overline{h}\|_{0}
+ C K^2 \tau^2 \|e^{\tau \varphi} \theta_{n}^{\frac{1-2s}{2}} u_1\|_{L^2}\\
&\leq
C \|e^{\tau \varphi} \theta_{n}^{\frac{2s-1}{2}} \overline{f}\|_{L^2}
+ C \tau^{1-s}\|e^{\tau \varphi}\overline{h}\|_{0}.
\end{split}
\end{align}
Here $\tilde{\chi}(t)$ denotes the characteristic function of the interval $(t_1,t_2)$.
Rewriting \eqref{eq:sub_ell} in Cartesian coordinates, then concludes the argument for the proposition.
\end{proof}

\bibliographystyle{alpha}

\end{document}